\documentclass[11pt]{article}
\usepackage{amsmath}
\usepackage{amssymb}
\usepackage{amsthm}
\usepackage{lineno}
\usepackage{graphicx}
\usepackage{caption}
\usepackage{subcaption}
\usepackage{blkarray}
\usepackage{appendix}
\usepackage{enumerate}
\usepackage[font=small,skip=0pt]{caption}
\usepackage{xcolor}
\usepackage{verbatim}
\usepackage{geometry}
\usepackage{enumerate}
\usepackage{wrapfig}
\usepackage{color,soul}
\usepackage{float}
\usepackage{lineno}
\usepackage[stable]{footmisc}
\usepackage{lineno}
\usepackage{blindtext}
\usepackage{calc}
\usepackage{upgreek}
\usepackage{amsmath}
\usepackage{amssymb}
\usepackage{amsthm}
\usepackage{lineno}
\usepackage{graphicx}
\usepackage{caption}
\usepackage{subcaption}
\usepackage[font=small,skip=0pt]{caption}
\usepackage{xcolor}
\usepackage{verbatim}
\usepackage{geometry}
\usepackage[colorinlistoftodos]{todonotes}
\usepackage{wrapfig}
\usepackage{color,soul}
\usepackage{float}
\usepackage{url}
\usepackage[colorinlistoftodos]{todonotes}
\usepackage{wrapfig}
\usepackage{color,soul}
\usepackage{float}
\usepackage{multirow}
\usepackage[utf8]{inputenc}

\numberwithin{equation}{section}
\newtheorem{Theorem}{Theorem}
\newtheorem{Corollary}{Corollary}

\newtheorem*{Note*}{Note}
\newtheorem{Proposition}{Proposition}
\newtheorem{Remark}{Remark}

\newtheorem*{Recall*}{Recall}
\newcommand{\geqs}{\geqslant}
\newcommand{\leqs}{\leqslant}
\graphicspath{{Figures/}}
\usepackage{color}

\definecolor{ao(english)}{rgb}{0.0, 0.5, 0.0}

\newcommand{\W}{\widetilde{\bold{W}}}
\newcommand{\Z}{\widetilde{\bold{Z}}}
\newcommand{\F}{\widetilde{\bold{F}}}
\newcommand{\G}{\widetilde{\bold{G}}}
\newcommand{\LMat}{\widetilde{\bold{L}}}

\usepackage[utf8]{inputenc}
\usepackage[titles]{tocloft}

\title{Exit Times for a Discrete Markov Additive Process}
\author{
        Zbigniew Palmowski$^{a,}$\footnote{Corresponding author e-mail address: zbigniew.palmowski@gmail.com}\,, \,  Lewis Ramsden$^{b,}$\footnote{Second author e-mail address: lewis.ramsden@york.ac.uk}\, and   Apostolos D. Papaioannou$^{c,}$\footnote{Third author  e-mail address: papaion@liverpool.ac.uk}     \\
      \\ $^a$Department of Applied Mathematics\\
       Wroc\l{}aw University of Science and Technology \\
       Wroc\l{}aw, Poland; \\
        $^b$ School for Business and Society, University of York\\
       York, Yorkshire, YO10 5DD, United Kingdom;\\
       $^c$Institute for Financial and Actuarial Mathematics \\
               Department of Mathematical Sciences\\
        University of Liverpool, L69 7ZL,  United Kingdom
       }

\begin{document}
\maketitle

\begin{abstract}
\noindent In this paper we consider (upward skip-free) discrete-time and discrete-space Markov additive chains (MACs) and develop the theory for the so-called $\W$ and $\Z$ scale matrices. which are shown to play a vital role in the determination of a number of exit problems and related fluctuation identities.  The theory developed in this fully discrete setup follows similar lines of reasoning as the analogous theory for Markov additive processes in continuous-time and is exploited to obtain the probabilistic construction of the scale matrices, identify the form of the generating function and produce a simple recursion relation for $\W$, as well as its connection with the so-called occupation mass formula. In addition to the standard one and two-sided exit problems (upwards and downwards), we also derive distributional characteristics for a number of quantities related to the one and two-sided `reflected' processes.

\end{abstract}

\noindent {\bf Keywords:} Markov Additive Process; Fluctuation Theory; Exit Problems; Discrete-Time; Scale Matrices; Random Walk.

\section{Introduction}

Exit problems for stochastic processes is a classic topic in applied probability and has received a great deal of attention within the literature. In the continuous setting (time and space), exit problems for so-called upward skip-free processes, known in the literature as `spectrally negative L\'evy processes', have been extensively considered in \cite{BERT} (Chapter VII), \cite{KYPR} (Chapter 8) and references therein, by means of fluctuation theory where semi-explicit expressions are derived in terms of the so-called `scale functions'. On the other hand, in the fully discrete setting exit problems for general discrete-time random walks are excellently treated in \cite{FEL} and \cite{MJ}, among others, by means of probabilistic arguments and include, as particular cases, the corresponding upward skip-free random walks. That is, a random walk for which downward jumps are unrestricted but upward jumps are constrained to a magnitude of at most one, emulating the upward `drift' in continuous-time. More recently, \cite{AV1} implement the ideas underlying the exit problems for continuous spectrally negative L\'evy processes for their discrete random walk counterparts and derive exit problems and other fluctuation identities in terms of analogous `discrete scale functions'.

A natural generalisation of the above processes are the broad family of Markov Additive Processes (MAPs), which incorporate an externally influencing Markov environment, providing greater flexibility to the characteristics of the underlying process in terms of its claim frequency and severity distributions, see \cite{ASM} (Chapter XI). Within this generalised framework, the existence of multidimensional scale functions, known as `scale matrices', was first discussed in \cite{KYPAL} and were used to derive fluctuation identities and first passage results for continuous-time MAPs. \cite{IVAN1} extended the initial findings of \cite{KYPAL} by providing the probabilistic construction of the scale matrices, identifying their transforms and considering an extensive study of exit problems including one-sided and two-sided exits, as well as exits for reflected processes via implementation of the occupation density formula. Further studies on MAPs and their exit/passage times can be found in \cite{BIKM}, \cite{BR}, \cite{DM}, among others. More recently, \cite{IVAN4}, derive and compare results for continuous-time MAPs with lattice (discrete-space) and non-lattice support. It is worth noting here that the authors in this work do discuss some of the corresponding results for the fully-discrete (time and space) MAP model considered in this paper, however, only a limited number of results are stated and a variety of important steps and proofs were omitted.

This paper bridges the gap between the aforementioned works and provides a  theoretical framework for fully discrete, upward skip-free MAPs, in terms of `discrete scale matrices', spelling out the differences in results, methodologies and necessary adjustments for deriving fluctuation identities between discrete and continuous MAPs. In particular, we derive results for the first passage theory, including one and two-sided exit problems as well as the under(over)-shoots upon exit via the associated `reflected' process. The motivation for deriving such a framework comes from the discrete set up having known advantages over the continuous-time models. For example, it is known that the Wiener-Hopf factorisation can be replaced by a simple Laurent series (see \cite{AV1}). Moreover, due to the equivalence between a discrete MAP and a Markov-modulated random walk, this paper provides a more flexible random walk model and enriches the numerous applications of random walk theory across a variety of disciplines.

The paper is organised as follows: In Section 2 we define the MAP in discrete time and space and derive the so-called occupation mass matrix formula, from which we obtain some useful identities to be used in the following sections. In Section 3, we introduce some fundamental matrices associated to the discrete MAP, identify the first of two discrete scale matrices and derive matrix expressions for the one and two-sided upward exit problem. In Section 4, we derive results for the corresponding one and two-sided reflected processes, including the over-shoot and under-shoot upon exit which are then used in Section 5 to derive expressions for the one and two-sided downward exit problems of the original (non-reflected) discrete MAP.

\section{Preliminaries} \label{Prem}

\noindent A fully discrete (time and space) MAP, which we will call a \textit{Markov Additive Chain} (MAC), is defined as a bivariate discrete-time Markov chain $(X,J) = \{(X_n, J_n)\}_{n \geqs 0}$, on the product space $\mathbb{Z}\times E$, where $X_n \in \mathbb{Z}$ describes the \textit{level} of the underlying process, whilst $J_n \in E = \left\{ 1, 2, \ldots, N \right\}$ describes the \textit{phase} of some external Markov chain (which affects the dynamics of $X_n$) having transition probability matrix $\bold{P}$, such that for $i,j \in E$, $\left( \bold{P} \right)_{ij} = \mathbb{P} \left( J_1 = j | J_{0} =i \right)$. It is assumed throughout this work, that the Markov chain $\{J_n\}_{n \geqs 0}$ is ergodic such that its stationary distribution $\boldsymbol{\pi}^\top = \left( \pi_1, \ldots, \pi_N\right)$ exists and is unique.
 The defining property of the MAC is the conditional independence and stationarity of law governing $X_n$, given $J_n$. That is, given that $\{ J_T = i \}$ for some fixed $T\in \mathbb{N}$, the Markov chain $\{ (X_{T+n} - X_T, J_{T+n}) \}$ is independent of $\mathcal{F}_T$ (the natural filtration to which the bivariate process $(X, J)$ is adapted) and $\{ (X_{T+n} - X_T, J_{T+n}) \} \overset{d}{=} \{ (X_n - X_0, J_n) \}$,
 given $\{J_0 = i\}$ for any phase state $i \in E$. This is known as the Markov additive property, a consequence of which is that the level process $\{X_n\}_{n \geqs 0}$ is translation invariant on the lattice.

Intuitively, the MAC is simply a Markov-modulated random walk where $\{X_n\}_{n \geqs 0}$ evolves according to the random walk
$X_n = Y_1 + Y_2 + \cdots + Y_n$, where $\{Y_k\}_{k \geqs 1}$ is a sequence of conditionally i.i.d.\,random variables with common, conditional distribution $\widetilde{q}_{ij}(y) = \mathbb{P}( Y_1 = y | J_1 = j, J_{0}=i),$ and thus, probability mass matrix $\widetilde{\bold{Q}}(y)$, with $i$,$j$-th element $\bigl(\widetilde{\bold{Q}}(y)\bigr)_{ij} = \widetilde{q}_{ij}(y)$. As such, and due to the invariance property, the transition probability matrix of $\left(X, J\right)$ has a block-like structure with blocks $\widetilde{\bold{A}}_m$ which represent the (one-step) transition matrix for an increase of $m$ levels in $\{X_n\}_{n \geqs 0}$ whilst capturing the phase transitions of $\{J_n\}_{n \geqs 0}$, such that
\begin{linenomath*}\begin{equation}
\label{Amat}
\widetilde{\bold{A}}_m = \widetilde{\bold{Q}}(m) \circ \bold{P},
\end{equation}\end{linenomath*}
where $\circ$ denotes entry-wise products (Hadamard multiplication).
In the remainder of this paper, we assume that $X=\{X_n\}_{n\geq 0 }$ may only increase by at most one level per unit time whilst experiencing downward jumps of arbitrary size. That is, for all $i,j \in E$, we have $\widetilde{q}_{ij}(m) = \mathbb{P}( Y_1 = m | J_1 = j, J_{0}=i) \geqs 0$ for $m \leqs 1$ and $\widetilde{q}_{ij}(m) = 0$ otherwise, which leads to $\widetilde{\bold{Q}}(m) = \bold{0}$ and thus $\widetilde{\bold{A}}_m = \bold{0}$ for $m = 2, 3, . \ldots$. In this sense, we say that $X$ possesses an `upward skip-free' property, an advantage of which is that the value of $X$ is known at stopping time corresponding to `upward' crossing/hitting of a given level (see below). This corresponds to the discrete analogue of a `spectrally negative' MAP in the continuous setting, which have important applications to workload and surplus processes in queuing and risk theory, respectively (see \cite{ASM} and \cite{RUIN} for more details).

\subsection{MAC Matrix Generator }
It has already been noted that the dynamics of the level process ($X$) within the MAC depends on the phase transitions of the external Markov chain ($J$). As such, the majority of quantities and results presented in this paper depend on the initial and final states of $\{J_n\}_{n \geqs 0}$ and thus, are given in matrix form. With this in mind, let us define the expectation matrix operator $\mathbb E_x(\cdot\,; J_n)$ which denotes an $N\times N$ matrix with $i,j$-th element $\left(\mathbb{E}_x\left(\cdot \, ; J_n\right)\right)_{ij}=\mathbb{E}\left( \cdot \,1_{( J_n = j)} | X_0 = x, J_0 = i  \right)$, where $1_{(\cdot)}$ represents the indicator function, and corresponding probability matrix $\mathbb{P}_x( \cdot \, , J_n)$ with elements $\left(\mathbb{P}_x( \cdot \, , J_n)\right)_{ij}=\mathbb{P}( \cdot \,, J_n = j | X_0 = x, J_0 = i)$.
Moreover, we denote $\mathbb{E}\left(\cdot \, ; J_n\right) \equiv \mathbb{E}_0\left(\cdot \, ; J_n\right)$, having associated probability measure $\mathbb{P}\left(\cdot \, ,J_n\right) \equiv \mathbb{P}_0\left(\cdot \, , J_n\right)$ and thus, we can define the probability generating matrix (p.g.m.)\,of the process $\{X_n\}_{n \geqs 0}$ with initial level $X_0=0$,  for at least $|z| \leqs 1$ and $z \neq 0$, by $\mathbb{E}\left( z^{-X_n}\, ; J_n \right)$, which satisfies
\begin{linenomath*}\begin{equation}\label{eqFMat}
 \mathbb{E}\left( z^{-X_n}\, ; J_n \right) = \bigl(\widetilde{\bold{F}}(z)\bigr)^n, \quad \widetilde{\bold{F}}(z) :=\mathbb{E}\left( z^{-X_1}\, ; J_1 \right) = \sum_{m=-1}^\infty z^{m}\widetilde{\bold{A}}_{-m},
\end{equation}\end{linenomath*}
and for $z=1$, we have $\widetilde{\bold{F}}(1) = {\bold{P}} = \sum_{m=-1}^\infty \widetilde{\bold{A}}_{-m} $.

\begin{Remark}
\label{rem1} Note that since the matrices $\widetilde{\bold{A}}_{-m}$ are probability transition matrices, such that $\widetilde{\bold{A}}_{-m} \geqs 0$ (non-negative),   it follows that for $z > 0$, the matrix $\widetilde{\bold{F}}(z)$ is also non-negative.  Hence, by the Perron-Frobenius theorem, $\widetilde{\bold{F}}(z)$ has a (simple) eigenvalue, denoted $\kappa(z)$, which is greater than or equal in absolute value than all other eigenvalues with corresponding left and right (column) eigenvectors, denoted $\vec{{\bold v}}(z)$ and $\vec{{\bold h}}(z)$, respectively, such that
$\vec{\bold{v}}(z)^\top \widetilde{\bold{F}}(z) = \kappa(z) \vec{\bold{v}}(z)^\top$ and $\widetilde{\bold{F}}(z) \vec{\bold h}(z) = \kappa(z) \vec{\bold h}(z)$. Moreover, since $\widetilde{\bold{F}}(1) = \bold{P}$ is a stochastic matrix, using standard facts from matrix analysis (see \cite{BIETAL}) we have $\kappa(1) = 1$ and it can be shown that $\kappa'(1)$ determines the asymptotic drift of the level process $\{X_n\}_{n \geqs 0}$ (see Section 1.3 in \cite{NEUT1} and \cite{GAIL}),  given by
\begin{linenomath*}\begin{align*}
\lim_{n \rightarrow \infty} \frac{X_n }{n} = - \kappa'(1) = \ -  \boldsymbol{\pi}^\top \sum_{m = -1}^\infty m\, \widetilde{\bold{A}}_{-m}\bold{e}.
\end{align*}\end{linenomath*}

\end{Remark}

\noindent Within the theory of continuous-time L\'evy processes, it is often desirable to analyse the process prior to some independent exponential `killing time' as this can emulate the role of Laplace transforms or generating functions within the calculations (see \cite{KYPR}). For a MAP, this exponential killing time can alternatively be incorporated via an enlargement to the state space of the Markov chain with the addition of an `absorbing' (killing) state and analysing the process prior to absorption (see \cite{IVAN1} for details).

In a similar way, let us enlarge the state space $E$ to $E \cup \{\dagger\}$, where $\dagger$ denotes an absorbing state, often called the \textit{cemetery} state, and we set $X=\partial$ whenever $J=\dagger$. Moreover, let us assume that the (one-step) `absorption' probability is the same from all states and denoted by $1-v = \mathbb{P}\left(J_1 = \dagger | J_0 = i \right)$, for all $i \in E$, such that the corresponding `non-absorption' probability (survival) is given by $v \in (0,1]$. Now, due to the addition of this cemetery state, it is clear that the probability transition matrix for transitions between the `transient' (when $v<1$) states of $E$ is dependent on $v$. Let us define this by $\bold{P}(v)\equiv v\bold{P}$, where $\bold{P}$ denotes the stochastic probability transition matrix defined in Section \ref{Prem}, in the absence of an absorbing state or `killing' ($v=1$). Hence, it follows that $\bold{P}(v)\vec{\bold e} = v\bold{P}\vec{\bold e} = v\vec{\bold e}$ and thus, for $v <1 $, $\bold{P}(v)$ is sub-stochastic and its Perron-Frobenius eigenvalue is less than 1 (see \cite{BIETAL}). Finally, it follows that the absorption or `killing' time of the Markov chain, denoted $g_v=\inf\{n> 0:J_n=\dagger\}$, is geometrically distributed with parameter $v\in (0,1]$ and we have
\begin{linenomath*}\begin{equation}\label{killingmainid}
\mathbb{E}\bigl( z^{-X_n}; n < g_{v} ,J_n \bigr) = {v}^n \widetilde{\bold{F}}(z)^n  = \bigl(v \widetilde{\bold{F}}(z)\bigr)^n = \bigl(\widetilde{\bold{F}}^{v}(z) \bigr)^n\end{equation}\end{linenomath*}
where $ \widetilde{\bold{F}}^{v}(z) := \mathbb{E}\bigl( z^{-X_1}; 1 < g_{v} ,J_1 \bigr) = {v}\widetilde{\bold{F}}(z)$ with $\widetilde{\bold{F}}(z)$ denoting the matrix generator of the MAC in the absence of killing, as defined as in Eq.\,\eqref{eqFMat}. The connection between the killed process and transforms/generating functions of the non-killed process is evident when we note that Eq.\,\eqref{killingmainid} is equivalent to $\mathbb{E}\bigl( v^n z^{-X_n}; J_n \bigr)$ for a `non-killed' MAC. Further advantages of working with the killed process are discussed in more details in later sections. Throughout the remainder of this paper, we generally suppress the explicit notation that absorption has not yet occurred but point out that it is assumed implicitly. As such, the results derived in the following are, in fact, much more general than they appear, with only a handful of these generalisations being stated explicitly.

\subsection{Occupation Times}

It is well known that occupation times and their densities play an important role within the theory of L\'evy processes and their fluctuations. In a continuous environment, the definition of the occupation density/time of a process at a given level has to be treated with some care and detail (see \cite{BERT}, \cite{IVAN1}) however, in the fully discrete model considered in this paper, the mathematical definition is intuitive.

Let us define by $\widetilde{L}(x, j, n)$, the \textit{occupation mass} denoting the number of periods the process $\{(X_n, J_n)\}_{n \geq 0}$ is in state $(x, j)\in \mathbb{Z} \times E$, up to and including time $n \geqs 0$, such that
\begin{linenomath*}\begin{equation}
\widetilde{L}(x, j, n) = \sum_{k = 0}^n 1_{\left( X_k = x, J_k = j \right)}.
\end{equation}\end{linenomath*}
Then, for some measurable non-negative function $f$, we have the so-called discrete \textit{occupation mass formula}
\begin{linenomath*}\begin{align}
\label{occdens}
\sum_{k=0}^n f\left(X_k\right) 1_{(J_k = j )} &= \sum_{x \in \mathbb{Z}} f\left(x\right) \widetilde{L}\left( x, j , n \right).
\end{align}\end{linenomath*}
From the above definition, it is clear that $\widetilde{L}(x, j, n)$ is a non-decreasing (monotone) process in $n \geqs 0$, which is adapted to the natural filtration $\mathcal{F}_n$. Therefore, if we further define the $N$-dimensional square occupation mass matrix, denoted $\widetilde{\bold{L}}(x, n)$, with $i,j$-th element given by $\bigl(\widetilde{\bold{L}}(x, n)\bigr)_{ij} = \mathbb{E}\bigl(  \widetilde{L}\left( x, j , n \bigr) \big| J_0 = i\right)$. Then, by application of the strong Markov property, analogously to Proposition 8 in \cite{IVAN1}, we have the following proposition.
\begin{Proposition}
\label{PropDens}
Let
\[\tau_x = \inf\{ n \geqs 0 : X_n = x \},\]
denote the first `hitting' time of the level $x\in \mathbb{Z}$. Then, for the occupation mass matrix  $\widetilde{\bold{L}}(x, n)$, it follows that
\begin{linenomath*}\begin{equation}
\widetilde{\bold{L}}(x, \infty) = \mathbb{P}\left( \tau_x < \infty, J_{\tau_x} \right) \widetilde{\bold{L}},
\end{equation}\end{linenomath*}
where $\left(\mathbb{P}\left( \tau_x < \infty, J_{\tau_x} \right)\right)_{ij}=\mathbb{P}\left( \tau_x<\infty, J_{\tau_x}=j|J_0=i \right)$ and $\widetilde{\bold{L}} := \widetilde{\bold{L}}(0, \infty)$ is the occupation mass matrix at the level $0$ over an infinite-time horizon, which has strictly positive entries.
\end{Proposition}

\begin{Remark} Let us point out some of the advantages of working with the killed process at this point:
\label{killing}
\begin{enumerate}[(i)]
\item If we include the implicit killing in the calculations explicitly, then for $v \in (0,1]$, the probability $ \mathbb{P}\left( \tau_x < \infty, J_{\tau_x} \right)$ becomes
\begin{linenomath*}\begin{align*}
\mathbb{P}\bigl( \tau_x < g_{v} , J_{\tau_x} \bigr)&=\sum_{n=0}^{\infty}\mathbb{P}(\tau_x = n, n <g_v, J_{n})=\sum_{n=0}^{\infty}v^n \mathbb{P}(\tau_x = n, J_n )= \mathbb{E}\bigl( {v}^{\tau_x};J_{\tau_x}\bigr),
\end{align*}\end{linenomath*}
where in the second equality we have used the fact that $\bold{P}(v) = v\bold{P}$. That is, the probability matrix $\mathbb{P}(\tau < \infty, J_{\tau_x})$ becomes the generator matrix $\mathbb{E}( {v}^{\tau_x};J_{\tau_x})$, if one imposes `killing' explicitly. As mentioned above, throughout this work we will keep killing implicit as it greatly simplifies the presentation but highlight that the above idea holds for all results.
\item Similarly, by superimposing killing in  Proposition \ref{PropDens}, we have that
\begin{linenomath*}\begin{equation*}
\widetilde{\bold{L}}_{v}(x, \infty) = \mathbb{E}\bigl({v}^{\tau_x}; J_{\tau_x} \bigr) \widetilde{\bold{L}}_{v},
\end{equation*}\end{linenomath*}
 with $\widetilde{\bold{L}}_{v} := \widetilde{\bold{L}}_{v}(0, \infty)$, such that the $i,j$-th element of $\widetilde{\bold{L}}_{v}(x, n)$ is given by
$\bigl(\widetilde{\bold{L}}_v(x, n)\bigr)_{ij} = \mathbb{E}\bigl( \widetilde{L}_v(x,j,n) \big| J_0=i\bigr)$,
where $\widetilde{L}_{v}(x,j,n) = \sum_{k=0}^n1_{\left(X_k =x, J_k=j, k<g_{v}  \right) }.$
Note that since $\bold{P}$ is sub-stochastic, then $\{X_k=x\}$ implies that $\{k<g_v\}$ and thus $\widetilde{L}_{v}(x,j,n)$ coincides with $\widetilde{L}(x,j,n)$.
\end{enumerate}
\end{Remark}


\noindent The main reason for introducing the theory of occupation times and their associated mass matrices, is due to their relationship with the one step p.g.m., namely $\widetilde{\bold{F}}(z)$. This connection is highlighted in the following auxiliary theorem which provides the foundations for many of the results in the following sections.
\begin{Theorem}
\label{TheoremOcc}
For all $z \in (0,1]$ such that $\bold{I} - \widetilde{\bold{F}}(z)$ is non-singular,
it follows that
\begin{linenomath*}\begin{equation}
\label{eqoccrel}
\sum_{x \in \mathbb{Z}} z^{-x} \mathbb{P}\left(\tau_x < \infty,  J_{\tau_x} \right)\widetilde{\bold{L}} = \bigl(\bold{I} - \widetilde{\bold{F}}(z) \bigr) ^{-1},
\end{equation}\end{linenomath*}
where $\tau_x$ is the first hitting time of the level $x \in \mathbb{Z}$.
\end{Theorem}
\begin{proof}
First note by the occupation mass  formula, that for any $j\in E$, we have
\begin{linenomath*}\begin{equation*}
\sum_{k=0}^{n} z^{-X_k}1_{(J_k=j)}=\sum_{x \in \mathbb{Z}} z^{-x}\widetilde{L}(x,j,n).
\end{equation*}\end{linenomath*}
Taking expectations in the above equation and considering the limit as $n \rightarrow \infty$, yields
\begin{linenomath*}\begin{equation*}
\lim_{n\rightarrow \infty}\sum_{k =0}^n \mathbb{E} \left(  z^{-X_k}1_{(J_k = j)} \big| J_0 = i \right)=\lim_{n \rightarrow \infty} \sum_{x \in \mathbb{Z}} z^{-x}\mathbb{E}\bigl(\widetilde{L}(x,j,n)\big| J_0 = i \bigr),
\end{equation*}\end{linenomath*}
from which, since $z^{-x}$ is non-negative for $z > 0$, we can apply the monotone convergence theorem to obtain
\begin{linenomath*}\begin{equation*}
\sum_{k =0}^\infty \mathbb{E} \left(  z^{-X_k}1_{(J_k = j)} \big| J_0 = i \right)=\sum_{x \in \mathbb{Z}} z^{-x}\mathbb{E}\bigl(\widetilde{L}(x,j,\infty)\big| J_0 = i \bigr).
\end{equation*}\end{linenomath*}
Equivalently, in matrix form the above expression can be written as
\begin{linenomath*}\begin{align}
\label{eqCONV}
\sum_{k =0}^\infty \widetilde{\bold{F}}(z)^k =\sum_{x \in \mathbb{Z}} z^{-x}\widetilde{\bold{L}}(x, \infty) =\sum_{x \in \mathbb{Z}} z^{-x}\mathbb{P}\left(\tau_x < \infty, J_{\tau_x} \right)\widetilde{\bold{L}}
\end{align}\end{linenomath*}
where the last equality comes from the result of Proposition \ref{PropDens}. Finally, we note that the geometric series on the l.h.s.\,converges to $(\bold{I}-\widetilde{\bold{F}}(z))^{-1}$ as long as $\kappa(z) < 1$ and the result follows using analytic continuation to extend the domain of convergence to all $z \in (0,1]$ such that $(\bold{I}-\widetilde{\bold{F}}(z))^{-1}$ exists. \end{proof}

\begin{Remark}
\label{RemInvert}
Note that the result of Theorem \ref{TheoremOcc} holds in the presence of killing ($v < 1$), since $\bold{P}(v)$ is sub-stochastic and thus $\kappa^v(1)<1$, where $\kappa^v(z)$ is the Perron-Frobenius eigenvalue of $\widetilde{\bold{F}}^v(z)$. Hence, by continuity of $\kappa^v(z)$, there exists a small interval around $z =1$ for which $\kappa^v(z) < 1$. In addition, $\widetilde{\bold{L}}$ must have finite entries as under killing the Markov chain is transient and the expected number of visits to any state is finite.
\end{Remark}



\section{Upward Exit Problems}

In this section we discuss and derive results on exit problems for upward skip-free MACs above and below a fixed level or strip. In the first instance, we will utilise the upward skip-free property of the level process, $\{X_n\}_{n \geqs 0}$, to determine expressions for upward exit times (one and two-sided), then extend the theory to consider downward exit problems. These expressions are given in terms of so-called \textit{fundamental} and \textit{scale matrices} associated to the MAC, where the existence of the latter were first discussed in \cite{KYPAL} and extend the notion of scale functions associated to L\'evy processes (see \cite{KYPR} and \cite{AV1} for more details).

All the results given in this section are stated from an initial level $X_0 = 0 $ which, due to the invariance property, can be generalised to an arbitrary level, say $x_0 \in \mathbb{Z}$, via an appropriate shift.

Let us denote by $\tau_x^{\pm}$, the first time the level process $\{X_n\}_{n \geqs 0}$ up(down)-crosses the level $x \in \mathbb{Z}$, such that
\begin{linenomath*}\begin{equation}
\label{crossingtime}
\tau_x^{+} = \inf \{ n \geqs 0: X_n \geqs x\} \quad \text{and} \quad \tau_x^{-} = \inf \{ n \geqs 0: X_n \leqs x\}.
\end{equation}\end{linenomath*}
We note that in a `spectrally negative' MAC with upward drift of one per unit time, for $x \geqs X_0$ the random stopping times $\tau^+_x$ (crossing time) and $\tau_x$ (hitting time)  coincide. Moreover, we have that $X_{\tau^+_x} = X_{\tau_x} = x$.

\subsection{One-Sided Exit Upward}
\label{oneside}
The key observation for the first passage upwards, is that the stationary and independent
increments as well as the skip-free property provide an embedded Markov structure. To see this, recall that $X_{\tau_1^+} = X_{\tau_1} = 1$, which together with the strong Markov and Markov additive properties, imply that the process $\{J_{\tau_n}\}_{n \geqs 0}$ is a (time-homogeneous) discrete-time Markov chain, given $X_0 = 0$, with some probability transition matrix $\widetilde{\bold{G}}$, such that for $a\geq0$,
\begin{linenomath*}\begin{equation}
\label{GMatrixGen}
\mathbb{P}\left(\tau_a < \infty, J_{\tau_a} \right) = \widetilde{\bold{G}}^a, \quad \widetilde{\bold{G}} = \mathbb{P}\left(\tau_1< \infty , J_{\tau_1}\right),
\end{equation}\end{linenomath*}
with $i,j$-th element given by $(\widetilde{\bold{G}})_{ij} = \mathbb{P}\left(\tau_1 < \infty, J_{\tau_1} = j \mid J_0 = i\right)$ for $i,j \in E$.
\begin{Remark} \label{RemG} In the case of no killing, i.e., $v=1$ and $\kappa'(1) \leqs 0$ (non-negative drift), the matrix $\widetilde{\bold{G}}$ is a stochastic matrix and sub-stochastic otherwise.
\end{Remark}

\noindent The transition probability matrix $\widetilde{\bold{G}}$ is widely known as the \textit{fundamental matrix} of the MAC and contains the probabilistic characteristics to determine upward passage times and the corresponding phase state at passage. That is, determining the matrix $\widetilde{\bold{G}}$ provides the probability of hitting any upper level $a \geqs 0$ and the phase of $\{J_n\}_{n \geqs 0}$ at this hitting time.

The matrix $\widetilde{\bold{G}} $ has a long history in the theory of structured stochastic matrices (see for e.g., Lemma 4.2 in \cite{BIETAL}) and can be computed by conditioning on the first time period, i.e.,
\begin{linenomath*}\begin{align*}
\widetilde{\bold{G}} = \mathbb{P}\left(\tau_1< \infty, J_{\tau_{ 1 }}\right)
&= \sum_{m=-1}^\infty \widetilde{\bold{A}}_{-m} \widetilde{\bold{G}}^{m+1} = \Bigl( \sum_{m=-1}^\infty \widetilde{\bold{A}}_{-m} \widetilde{\bold{G}}^m \Bigr)\widetilde{\bold{G}}.
\end{align*}\end{linenomath*}
Multiplying on the right by $\widetilde{\bold{G}}^{-1}$, assuming it exists (see Remark  \ref{Reminv}), and using the definition of $\widetilde{\bold{F}}(z)$ given in Eq.\,\eqref{eqFMat}, it follows that the fundamental matrix, $\widetilde{\bold{G}}$, is the right solution of $\widetilde{\bold{F}}(\cdot) = \bold{I}$, which is a well known equation established in \cite{NEUT} and further studied in \cite{BIETAL}, \cite{NEUT1},  \cite{GAIL} and   \cite{GAIL1}, among others.

\begin{Remark}\label{remkil}
Let us discuss a few important observations about the fundamental matrix $\widetilde{\bold{G}}$ and its significance within applied probability:	
\begin{enumerate}[(i)]
\item
For the continuous-time (scalar) spectrally negative L\'evy process, the fundamental matrix $\widetilde{\bold{G}}$, corresponds to the of inverse Laplace exponent at zero, namely $\Phi(0)$, i.e., the solution to $\psi(\beta) = 0$, where $\psi(\beta)$ denotes the Laplace exponent of the L\'evy process (see \cite{KYPR}).
\item It follows by definition that $\mathbb{E}\bigl(\widetilde{\bold{G}}^{-X_n}; J_n \bigr)$ is a martingale. In fact, it is clear that in the matrix setting, there exists another solution (left solution) to $\widetilde{\bold{F}}(\cdot) = \bold{I}$, say $\widetilde{\bold{R}}$, which would also result in the martingale $\mathbb{E}\bigl(\widetilde{\bold{R}}^{-X_n}; J_n \bigr)$. It turns out that the matrix $\widetilde{\bold{R}}$ is actually the counterpart of $\widetilde{\bold{G}}$ for the `time-reversed MAC' and is considered another \textit{fundamental matrix}. The time-reversed MAP and the corresponding matrix $\widetilde{\bold{R}}$ are considered in \cite{IVAN2} for the continuous-time (lattice) case and we direct the reader to this paper for more details.

\item Superimposing killing in the above produces the transform of the first passage time, namely \,$\mathbb{E}\left( {v}^{\tau_a} ; J_{\tau_a} \right)$, such that
\begin{linenomath*}\begin{equation}
\label{GvMatrix}
\mathbb{E}\bigl( {v}^{\tau_a} ; J_{\tau_a} \bigr) = \widetilde{\bold{G}}_{v}^{a}, \quad \widetilde{\bold{G}}_{v} = \mathbb{E}\bigl( {v}^{\tau_1} ; J_{\tau_1} \bigr),
\end{equation}\end{linenomath*}
 and   $\widetilde{\bold{G}}_{v}$ is the right solution of $\widetilde{\bold{F}}(\cdot) = {v}^{-1}\bold{I}$.
 \item As discussed in \cite{IVAN2}, the right solutions of the above equations cannot be determined analytically except in some special cases. However, there exists a number of numerical algorithms which can be employed, e.g., the iterative algorithm \cite{NEUT}, logarithmic reduction \cite{LRAM} and the cyclic reduction \cite{BIME}, to name a few. For further details on the variety of algorithms available for solving such equations, see \cite{BIETAL} and references therein.
\end{enumerate}
\end{Remark}

\noindent

 \subsection{Two-Sided Exit Upward - $\{ \tau_a^+ < \tau_{-b}^-\}$ }

Within the literature of spectrally negative L\'evy processes and their fully discrete counterparts \cite{AV1}, the common approach to solving two-sided exit problems relies on the introduction of a family of functions, $W^q$ and $Z^q$, known as the $q$-scale functions (see \cite{KYPR} for details). The extension of these auxiliary, one dimensional scale functions to the multidimensional MAP setting was first proposed in \cite{KYPAL}, where the existence of the corresponding `scale matrices' was shown and were further investigated in \cite{IVAN1} who derived their probabilistic interpretation within the continuous setting.

For ${v} \in (0, 1]$,
 the discrete $\W_{v}$ scale matrix is defined as the mapping $\W_{v} : \mathbb{N} \rightarrow \mathbb{R}^{N \times N}$, with $\W_{v}(0) = \bold{0}$ (the matrix of zeros), such that
\begin{linenomath*}\begin{equation}
\label{eqW1}
\W_{v}(n) = \Bigl[\widetilde{\bold{G}}_{v}^{-n} - \mathbb{E}\bigl({v}^{\tau_{-n}}; J_{\tau_{-n}} \bigr) \Bigr]\widetilde{\bold{L}}_{v},
\end{equation}\end{linenomath*}
where we write $\W_1(n) =: \W(n)$ for the $1$-scale matrix. The definition of the scale matrix above is only unique up to a multiplicative constant and the presence of the infinite-time occupation matrix, $\widetilde{\bold{L}}_{v}$, is somewhat arbitrary here but is included in order to obtain the most concise form for the p.g.m.\,of $\W_{v}(\cdot)$, which is derived in Theorem \ref{ThmTwoSideUp} (see also \cite{IVAN1}).

In the two-sided exit problem, we are interested in the time of exiting a (fixed) `strip', $[-b,a]$, consisting of an upper and lower level denoted by $a$ and $-b$, respectively, such that $a > 0 > -b$. More formally, we are interested in the events $\{ \tau_a^+ < \tau_{-b}^-\,\}$ and $\{ \tau_a^+ > \tau_{-b}^-\,\}$, which correspond to the upward and downward exits from the strip $[-b,a]$, respectively. In this section, we are concerned with the former (upward exit). The the latter (downward exit) will be discussed in a later section as its derivation depends on alternative methods.

Let us denote by $\rho(\cdot)$, the spectral radius of a matrix. That is, if $\Lambda(\bold{A})$ denotes the spectrum of  a matrix $\bold{A}$, then $\rho\left(\bold{A}\right) = \max\{ |\lambda_i|: \lambda_i \in \Lambda(\bold{A})\}$.

\subsubsection{Two-sided exit theory for non-singular $\widetilde{\bold{A}}_1$}
 \begin{Theorem}
 \label{ThmTwoSideUp}
Assume that $\widetilde{\bold{A}}_1$ is non-singular. Then, there exists a  matrix $\widetilde{\bold{W}}: \mathbb{N} \rightarrow \mathbb{R}^{N\times N}$ with $\widetilde{\bold{W}}(0) = \bold{0}$, which is invertible and satisfies
\begin{linenomath*}\begin{equation}\label{twoexitprob}
\mathbb{P}\bigl( \tau^+_a < \tau^-_{-b}\,, J_{\tau^+_a} \bigr) = \widetilde{\bold{W}}(b)\widetilde{\bold{W}}(a+b)^{-1},
\end{equation}\end{linenomath*}
where $\bigl(\mathbb{P}\bigl( \tau^+_a < \tau^-_{-b}\,, J_{\tau^+_a} \bigr)\bigr)_{ij}=\mathbb{P}\bigl( \tau^+_a < \tau^-_{-b}\,, J_{\tau^+_a}=j|J_0=i \bigr)$ and
$\widetilde{\bold{W}}(\cdot)$ has representation
\begin{linenomath*}\begin{equation}
\label{ScaleMat1}
\widetilde{\bold{W}}(n)=\left(\widetilde{\bold{G}}^{-n} - \mathbb{P}\left(\tau_{-n}< \infty, J_{\tau_{-n}} \right)\right)\widetilde{\bold{L}}.
\end{equation}\end{linenomath*}
Furthermore, it holds that
\begin{linenomath*}\begin{equation}
\label{eqWTran}
\sum_{n=0}^\infty z^n \widetilde{\bold{W}}(n) = \Bigl( \widetilde{\bold{F}}(z) - \bold{I}  \Bigr)^{-1},
\end{equation}\end{linenomath*}
for  $z \in (0, 1]$ such that $z\notin \Lambda(\widetilde{\bold{G}})$, and
\begin{linenomath*}\begin{equation}\label{wld}
\widetilde{\bold{W}}(n) = \widetilde{\bold{G}}^{-n} \widetilde{\bold{L}}^+(n),
\end{equation}\end{linenomath*}
where $\widetilde{\bold{L}}^+(n) := \mathbb{E}\left[\widetilde{\bold{L}}\left(0, \tau_{n}\right)\right]$, denotes the expected number of times the process visits 0 before hitting level $n \in \mathbb{N}^+$.
\end{Theorem}

\begin{proof}
Following the same line of logic  as in \cite{IVAN1}, we note that  the events $\{ \tau_a^+ < \tau_{-b}^-\}$ and $\{ \tau_{ a}< \tau_{- b}\}$ are equivalent due to the upward skip-free property of $\{X_n\}_{n \geqs 0}$. This follows from the fact that in order to drop below $-b$ and then hit $a$, the process must visit $-b$ on the way. Thus, conditioning on possible events and employing the Markov additive property, we obtain
\begin{linenomath*}\begin{equation*}
\mathbb{P}\bigl(\tau_a < \infty, J_{\tau_{a}}  \bigr) = \mathbb{P}\bigl(\tau_{ a}< \tau_{ -b} ,J_{\tau_{a}}  \bigr) +  \mathbb{P}\bigl(\tau_{ a}> \tau_{ -b} , J_{\tau_{-b}} \bigr)\mathbb{P}\left( \tau_{a+b}< \infty, J_{\tau_{a+b}} \right)
\end{equation*}\end{linenomath*}
and
\begin{linenomath*}\begin{equation*}
\label{eq2}
\mathbb{P}\bigl( \tau_{ -b} < \infty, J_{\tau_{-b}} \bigr) = \mathbb{P}\bigl(\tau_{ a}> \tau_{ -b} , J_{\tau_{-b}}  \bigr) +  \mathbb{P}\bigl(\tau_{ a}< \tau_{ -b} ,J_{\tau_{a}} \bigr)\mathbb{P}\bigl(\tau_{-(a+b)} < \infty, J_{\tau_{-(a+b)}} \bigr).
\end{equation*}\end{linenomath*}
Now, by recalling that   $\mathbb{P}\bigl(\tau_a < \infty, J_{\tau_{a}}  \bigr)=\widetilde{\bold{G}}^a$, solving the second equation w.r.t.\,$\mathbb{P}\bigl(\tau_{ a}> \tau_{ -b} , J_{\tau_{-b}}  \bigr) $ and substituting the resulting equation into the first, yields
\begin{linenomath*}\begin{equation}
		\label{eqV1}
\mathbb{P}\bigl(\tau_{ a}< \tau_{ -b} , J_{\tau_{a}} \bigr)\Bigl[\mathbb{P}\left(\tau_{-(a+b)}< \infty, J_{\tau_{-(a+b)}} \right)\widetilde{\bold{G}}^{a+b} - \bold{I}\Bigr] = \mathbb{P}\left(\tau_{-b}< \infty, J_{\tau_{-b}} \right)\widetilde{\bold{G}}^{a+b} -\widetilde{\bold{G}}^{a}.
\end{equation}\end{linenomath*}
Finally, by multiplying through by $-\widetilde{\bold{G}}^{-(a+b)}$ on the right, we have
\begin{linenomath*}\begin{equation*}
 \mathbb{P}\bigl(\tau_{ a}< \tau_{ -b} , J_{\tau_{a}} \bigr)\left[ \widetilde{\bold{G}}^{-(a+b)} -\mathbb{P}\left(\tau_{-(a+b)}< \infty, J_{\tau_{-(a+b)}} \right) \right]= \widetilde{\bold{G}}^{-b} - \mathbb{P}\bigl(\tau_{-b}<\infty, J_{\tau_{-b}} \bigr),
\end{equation*}\end{linenomath*}
or equivalently
\begin{linenomath*}\begin{equation*}
\mathbb{P}\bigl(\tau_{ a}< \tau_{ -b} , J_{\tau_{a}} \bigr) = \widetilde{\bold{W}}(b) \widetilde{\bold{W}}(a+b)^{-1},
\end{equation*}\end{linenomath*}
given that  $\widetilde{\bold{W}}(\cdot)^{-1}$ exists (see Remark \ref{Reminv}). Note that the above result is derived in the absence of the occupation mass matrix, $\widetilde{\bold{L}}$, within the definition of $\W(n)$, reinforcing the point that the scale matrix is uniquely defined up to a (matrix) multiplicative constant. The choice for including $\widetilde{\bold{L}}$ in the definition of $\W(n)$, which is only well defined as long as $\widetilde{\bold{L}}$ has finite entries (see Remark \ref{RemInvert} for conditions), will become apparent in the following.

To prove Eq.\,\eqref{eqWTran}, let us take the transform of the scale matrix and recall the definition given in Eq.\,\eqref{ScaleMat1}, to obtain
\begin{linenomath*}\begin{equation}
\label{Wpgf}
\sum_{n=0}^\infty z^n \widetilde{\bold{W}}(n) = \sum_{n=0}^\infty z^n \widetilde{\bold{G}}^{-n}\widetilde{\bold{L}}- \sum_{n=0}^\infty z^n \mathbb{P}\left(\tau_{-n}<\infty, J_{\tau_{-n}} \right)\widetilde{\bold{L}},
\end{equation}\end{linenomath*}
where the first term on the r.h.s.\,satisfies
 \begin{linenomath*}\begin{equation}
\label{Wpgf1}
\sum_{n=0}^\infty z^n \widetilde{\bold{G}}^{-n}\widetilde{\bold{L}} = \sum_{n=0}^\infty \Bigl(z \widetilde{\bold{G}}^{-1}\Bigr)^n \widetilde{\bold{L}}= \left( \bold{I} - z \widetilde{\bold{G}}^{-1} \right)^{-1}\widetilde{\bold{L}},
\end{equation}\end{linenomath*}
 for all $z \in (0, \gamma)$, where $\gamma := \min\{|\lambda_i| : \lambda_i \in \Lambda(\widetilde{\bold{G}})\}$.

For the second term of Eq.\,\eqref{Wpgf}, under the conditions of  Theorem \ref{TheoremOcc}, we have
\begin{linenomath*}\begin{align*}
\bigl(\bold{I} - \widetilde{\bold{F}}(z) \bigr) ^{-1} &= \sum_{n=0}^\infty z^n \mathbb{P}\bigl(\tau_{-n} < \infty, J_{\tau_{-n}} \bigr)\widetilde{\bold{L}}+ \sum_{n=1}^\infty z^{-n} \mathbb{P}\bigl( \tau_{n}< \infty, J_{\tau_{n}} \bigr)\widetilde{\bold{L}}\\
&= \sum_{n=0}^\infty z^n \mathbb{P}\bigl(\tau_{-n}<\infty, J_{\tau_{-n}} \bigr)\widetilde{\bold{L}}+ \sum_{n=1}^\infty z^{-n} \widetilde{\bold{G}}^n\widetilde{\bold{L}}
\end{align*}\end{linenomath*}
for all $z \in (0,1]$ such that $\bigl(\bold{I} - \widetilde{\bold{F}}(z) \bigr) ^{-1}$ exists.
Moreover, for $z \in (\rho(\widetilde{\bold{G}}), 1]$ ($\rho(\widetilde{\bold{G}}) < 1 $  is true
as long as $\widetilde{\bold{G}}$ is invertible and this follows from the assumption that the matrix $\widetilde{\bold{A}}_1$ is non-singular, see also  Remark \ref{Reminv} ),
the geometric series on the r.h.s.\,converges and the above equation can be re-written as

\begin{linenomath*}\begin{eqnarray}
\label{ap4}
\bigl(\bold{I} - \widetilde{\bold{F}}(z) \bigr) ^{-1}  &=& \sum_{n=0}^\infty z^n \mathbb{P}\bigl( \tau_{-n}<\infty, J_{\tau_{-n}} \bigr)\widetilde{\bold{L}}+\Bigl( \bigl(\bold{I}-z^{-1}\widetilde{\bold{G}}\bigl)^{-1}-\bold{I} \Bigr)\widetilde{\bold{L}}\notag \\
      &=& \sum_{n=0}^\infty z^n \mathbb{P}\bigl(\tau_{ -n}< \infty, J_{\tau_{-n}} \bigr)\widetilde{\bold{L}}-\bigl( \bold{I} - z \widetilde{\bold{G}}^{-1} \bigr)^{-1}\widetilde{\bold{L}},
\end{eqnarray}\end{linenomath*}
once we prove a common domain of convergence, i.e., $\bigl(\bold{I} - \widetilde{\bold{F}}(z) \bigr) ^{-1}$ exists for some $z \in (\rho(\widetilde{\bold{G}}), 1]$. In fact, for $\rho(\widetilde{\bold{G}}) < 1 $, see  Lemma 4 in  \cite{BIKM}, it can be shown that the zeros of $\det[\bold{I} - \widetilde{\bold{F}}(z)]$ coincide with the eigenvalues of $\widetilde{\bold{G}}$ for $z \in (0,1]$ and thus, the above holds. 

Now, note that if we multiply Eq.\,\eqref{ap4} from the left by $\bold{I} - z \widetilde{\bold{G}}^{-1} $ and from the right by $\bold{I} - \widetilde{\bold{F}}(z)$, then both sides of the resulting equation are analytic for $z\in (0,1]$. Hence, since the matrices $\bigl( \bold{I} - z \widetilde{\bold{G}}^{-1} \bigr)$ and  $\bigl(\bold{I} - \widetilde{\bold{F}}(z)\bigr)$ are invertible as long as $z \notin \Lambda(\widetilde{\bold{G}})$ and thus for $z \in (0, \gamma)$, the aforementioned multiplication can be reversed and Eq.\,\eqref{ap4} holds for $z \in (0,\gamma)$ by analytic continuation. The result follows by substituting the above equation, along with Eq.\,\eqref{Wpgf1}, into Eq.\,\eqref{Wpgf} and using analytic continuation to extend the domain from $z \in (0,\gamma)$ to $z \in (0,1]$ such that $z \notin \Lambda(\widetilde{\bold{G}})$. 

To prove Eq.\,\eqref{wld}, we use similar arguments to those used for the result of Proposition \ref{PropDens}, to show that for $n \geqs 0$
\begin{linenomath*}\begin{equation} \label{eqV2}
\widetilde{\bold{L}} =\widetilde{\bold{L}}^+(n) + \mathbb{P}\left(\tau_n < \infty, J_{\tau_{ n}} \right) \mathbb{P}\left(\tau_{-n}< \infty, J_{\tau_{ -n}} \right) \widetilde{\bold{L}},
\end{equation}\end{linenomath*}
where $\widetilde{\bold{L}}^+(n): = \mathbb{E}\bigl(\widetilde{\bold{L}}(0, \tau_{n})\bigr)$, from which it follows that
\begin{linenomath*}\begin{equation*}
\widetilde{\bold{L}}^+(n) = \Bigl[\bold{I} - \mathbb{P}\left(\tau_n < \infty, J_{\tau_{ n}} \right) \mathbb{P}\left(\tau_{-n}< \infty, J_{\tau_{ -n}} \right) \bigr)\Bigr] \widetilde{\bold{L}} =\Bigl[\bold{I} - \widetilde{\bold{G}}^n \mathbb{P}\left(\tau_{-n} < \infty, J_{\tau_{ -n}} \right)\Bigr] \widetilde{\bold{L}}.
\end{equation*}\end{linenomath*}
Multiplying this expression through by $\widetilde{\bold{G}}^{-n}$ (on the left) and recalling the form of $\widetilde{\bold{W}}(n)$ given in Eq.\,\eqref{ScaleMat1}, the result follows immediately.
So far we assume only that $\rho(\widetilde{\bold{G}}) < 1 $, hence by Remark \ref{RemG}
that either $v<1$ or that $v=1$ and $\kappa^\prime (0)>0$.
To handle the remaining (limiting) case of $v=1$ and $\kappa^\prime (0)\leq 0$
we can follow the proof of Theorem 1 in
\cite{IVAN1}. Namely
we can use the representation \eqref{wld} of the scale function,
take $ v \rightarrow 1$ and observe that matrices $\widetilde{\bold{G}}$, $\widetilde{\bold{L}}^+(n)$ and $\widetilde{\bold{F}}(z)$
properly converge.

\end{proof}

\begin{Remark}
In \cite{IVAN2} the authors derive an equivalent result to Theorem \ref{ThmTwoSideUp} for a continuous-time MAP in the lattice and non-lattice case. Although their study focuses purely  on the continuous-time case, they do point out the connection for the discrete-time model (Remark 6 in \cite{IVAN2}) but do not provide any proof or further details.
\end{Remark}

 \begin{Remark}[Invertibility  of  $\widetilde{\bold{L}}^+(n)$, $\widetilde{\bold{G}}$ and $\widetilde{\bold{W}}(n)$]
\noindent Throughout the proof of the previous theorem and results earlier in this paper, we required invertibility of the fundamental matrix $\widetilde{\bold{G}}$ and the scale matrix $\W(n)$. We will now look at under what conditions such invertibility holds:

\label{Reminv}
\begin{enumerate}[(i)]
\item Following similar arguments as in \cite{IVAN2}, since the level process starts at $X_0 =0$, the expected number of visits at 0 before the process reaches level $n\geqs 0$, namely $\widetilde{\bold{L}}^+(n)=\mathbb{E}\bigl[\widetilde{\bold{L}}(0,\tau_{n})\bigr]$, satisfies
\begin{linenomath*}\begin{equation*}
\widetilde{\bold{L}}^+(n)=\bold{I}+\bold{\Pi}_n\,\widetilde{\bold{L}}^+(n),
\end{equation*}\end{linenomath*}
where $\bold{\Pi}_n$ is a probability matrix with $i,j$-th element containing the probability of a second visit to level 0 before reaching level $n$ and doing so in phase $j$, conditioned on the starting point $(0,i)$. Note that  $\bold{\Pi}_n$ is clearly a sub-stochastic, non-negative matrix, which implies $\rho\bigl(\bold{\Pi}_n\bigr)<1$ and thus $\bold{I}-\bold{\Pi}_n$ is invertible. Hence, $\widetilde{\bold{L}}^+(n)$ is also invertible, since from the above expression it follows that $\left( \bold{I} - \bold{\Pi}_n\right)\widetilde{\bold{L}}^+(n)=\bold{I}$.
\item
In order to show that $\widetilde{\bold{G}}$ is invertible, recall  that
\begin{linenomath*}\begin{equation*}
\widetilde{\bold{G}} = \sum_{m=-1}^\infty \widetilde{\bold{A}}_{-m} \widetilde{\bold{G}}^{m+1} =\widetilde{\bold{A}}_{1} +\sum_{m=0}^\infty \widetilde{\bold{A}}_{-m} \widetilde{\bold{G}}^{m+1},
\end{equation*}\end{linenomath*}
from which it follows that
\begin{linenomath*}\begin{align*}
\widetilde{\bold{A}}_{1} =\widetilde{\bold{G}}- \sum_{m=0}^\infty \widetilde{\bold{A}}_{-m} \widetilde{\bold{G}}^{m+1}&=\Bigl(\bold{I}-\sum_{m=0}^\infty \widetilde{\bold{A}}_{-m} \widetilde{\bold{G}}^{m}\Bigr)\bold{G}=\Bigl(\bold{I}-\bold{\Pi}_1\Bigr)\bold{G}
\end{align*}\end{linenomath*}
Therefore, since $\bold{I}-\bold{\Pi}_1$ is invertible, $\widetilde{\bold{G}}$ is invertible provided that $\widetilde{\bold{A}}_{1}$ is invertible. Finally, since $\widetilde{\bold{L}}^+(n)$ is invertible and given $\widetilde{\bold{G}}$ is invertible, then by Eq.\,\eqref{wld} it is clear that $\widetilde{\bold{W}}(n)$ is also invertible.
\end{enumerate}
\end{Remark}

\noindent Although Theorem \ref{ThmTwoSideUp} provides a number of representations for $\widetilde{\bold{W}}$, in the discrete case the scale matrix also satisfies a recursive relation. The recursion below generalises the recursion  for the scale function derived in \cite{AV1} and has also been discussed in \cite{IVAN2}.
\begin{Corollary}
	\label{CorW}
For $b\geqs 1$, the scale matrix $\widetilde{\bold{W}}(\cdot)$, defined in Theorem \ref{ThmTwoSideUp}, satisfies the following recursive equation
\begin{equation}
\label{EqRecur}
\widetilde{\bold{W}}(b+1) =  \widetilde{\bold{A}}_1^{-1} \Bigl(\widetilde{\bold{W}}(b) - \sum_{m=0}^{b-1} \widetilde{\bold{A}}_{-m} \widetilde{\bold{W}}(b-m) \Bigr),
\end{equation}
with $\widetilde{\bold{W}}(1)=\widetilde{\bold{A}}_1^{-1}$.
\end{Corollary}

\begin{proof}
To prove the recursive relation, consider the two-sided hitting probability $\mathbb{P}\bigl(\tau_a^+ < \tau_{-b}^-; J_{\tau^+_a} \bigr)$ and condition on the first time step. Then, for $a,b \geqs 1$, we have
\begin{align*}
\mathbb{P}\bigl(\tau_a^+ < \tau_{-b}^-; J_{\tau^+_a} \bigr) &= \sum_{m = -(b-1)}^1 \widetilde{\bold{A}}_m \mathbb{P}_m\bigl(\tau_a^+ < \tau_{-b}^-; J_{\tau^+_a} \bigr) \\
&=   \sum_{m = -(b-1)}^1 \widetilde{\bold{A}}_m \mathbb{P}\bigl(\tau_{a-m}^+ < \tau_{-(b+m)}^-; J_{\tau^+_{a-m}} \bigr),
\end{align*}
where the last equality follows from the Markov additive property.  Further, using   Theorem \ref{ThmTwoSideUp} and multiplying on the right by $\widetilde{\bold{W}}(a+b)$, the above expression can be re-written as
\begin{equation*}
\widetilde{\bold{W}}(b) = \sum_{m = -(b-1)}^1 \widetilde{\bold{A}}_m \widetilde{\bold{W}}(b+m).
\end{equation*}
and the recursive expression given in Eq.\,\eqref{EqRecur} follows directly after some basic algebraic manipulations.
For $\widetilde{\bold{W}}(1)$, recall Remark \ref{Reminv}  that $\widetilde{\bold{L}}^+(1)=(\bold{I}-\bold{\Pi}_1)^{-1}
$ and also  that $\widetilde{\bold{A}}_{1}^{-1} =\widetilde{\bold{G}} ^{-1}(\bold{I}-\bold{\Pi}_1)^{-1}=\widetilde{\bold{G}} ^{-1}\widetilde{\bold{L}}^+(1)=\widetilde{\bold{W}}(1)
$, from  Theorem \ref{ThmTwoSideUp}. \end{proof}

\begin{Remark} Under the same line of logic as Remark \ref{remkil}, we recall that the above results are more general than explicitly stated. For example, by superimposing killing Eq.\,\eqref{twoexitprob} is equivalent to
\begin{linenomath*}\begin{equation}
\mathbb{E}\Bigl({v}^{\tau^+_a} ; \tau^+_a < \tau^-_{-b}, J_{\tau^+_a} \Bigr) = \widetilde{\bold{W}}_{v}(b)\widetilde{\bold{W}}_{v}(a+b)^{-1},
\end{equation}\end{linenomath*}
for $v \in (0,1]$, where $\W_v(\cdot)$ is defined in Eq.\,\eqref{eqW1} with the rest of the results amended accordingly.
\end{Remark}

\subsubsection{Two-sided exit theory for arbitrary $\widetilde{\bold{A}}_1$  }

In Theorem \ref{ThmTwoSideUp}, we rely on the fact that $\widetilde{\bold{A}}_1$ is non-singular, which in turn ensures $\G$ is non-singular by Remark \ref{Reminv}. However, it turns out that a similar result can also be derived for arbitrary $\widetilde{\bold{A}}_1$ in terms of matrices closely related to the $\W$ scale matrix.

To see this, let us define $\widetilde{\bold{L}}^-(n): = \mathbb{E}\bigl(\widetilde{\bold{L}}\left(0, \tau_{-n}\right)\bigr)$ for $n \geqs 0$,  $\widetilde{\bold{M}}(n):=\mathbb{E}\bigl(\widetilde{\bold{L}}\left(-n, \tau_{-(n+1)}\right)\bigr)$ and recall $\widetilde{\bold{R}}$ is related to the `time-reversed' counterpart of $\widetilde{\bold{G}}$ (see Remark \ref{remkil}). Then, we have the following theorem.

\begin{Theorem}
	\label{ThmTwoSideUpSing}

Assume the matrix $\widetilde{\bold{A}}_{1}$ is singular. Then, there exists a  matrix $\widetilde{\bold{V}}: \mathbb{N} \rightarrow \mathbb{R}^{N\times N}$ with $\widetilde{\bold{V}}(0) = \bold{I}$, which is invertible and satisfies
\begin{linenomath*}\begin{equation}\label{twosidedsingular}
				\mathbb{P}\bigl(\tau^+_{ a}< \tau^-_{ -b} , J_{\tau^+_{a}} \bigr) = \widetilde{\bold{V}}(b) \widetilde{\bold{R}}^{a}\widetilde{\bold{V}}(a+b)^{-1},
		\end{equation}\end{linenomath*}
		where
		\begin{linenomath*}\begin{equation*}
				\widetilde{\bold{V}}(n)=\widetilde{\bold{L}}^-(n)=\Bigl[\bold{I} - \mathbb{P}\left(\tau_{-n}<\infty, J_{\tau_{-n}} \right) \widetilde{\bold{G}}^n \Bigr] \widetilde{\bold{L}}.
		\end{equation*}\end{linenomath*}
		Furthermore, it holds that
		\begin{linenomath*}\begin{equation}\label{WVofn}
				\widetilde{\bold{L}}^-(n)=\sum_{k=-1}^{n-1}\widetilde{\bold{M}}(k)\widetilde{\bold{R}}^k
		\end{equation}\end{linenomath*}
		and for $z \in (0,1]$ such that $z \notin \Lambda(\G)$, also
		\begin{linenomath*}\begin{equation}\label{WVofn2}
				\sum_{n=0}^\infty z^n \widetilde{\bold{M}}(n) = \left(\bold{I} - \widetilde{\bold{F}}(z)\right)^{-1}\Bigl(\bold{I} -  z^{-1}\widetilde{\bold{R}}   \Bigr).
		\end{equation}\end{linenomath*}
\end{Theorem}

\begin{proof}
	
	Assume now that the matrix $\widetilde{\bold{G}}$ is singular (which, by Remark \ref{Reminv}, is
		equivalent to the requirement that the matrix $\widetilde{\bold{A}}_{1}$ is singular).
		Then, from equation \eqref{eqV1} we can obtain an alternative representation for the two-sided exit probability of the form
		\begin{linenomath*}\begin{equation*}
				\mathbb{P}\bigl(\tau^+_{ a}< \tau^-_{ -b} , J_{\tau^+_{a}} \bigr) = \widetilde{\bold{H}}(b) \widetilde{\bold{G}}^a\widetilde{\bold{H}}(a+b)^{-1},
		\end{equation*}\end{linenomath*}
		where
		\begin{linenomath*}\begin{equation*}
				\widetilde{\bold{H}}(n)=\bold{I}-\mathbb{P}\left(\tau^-_{-n}< \infty, J_{\tau^-_{-n}} \right)\widetilde{\bold{G}}^{n},
			\end{equation*}
		\end{linenomath*}
		for $n\geqs 0$, as long as this matrix is invertible (see below).
		Moreover, by similar arguments as in Eq.\,\eqref{eqV2}, it follows that
		\begin{linenomath*}\begin{equation*}
				\widetilde{\bold{L}} =\widetilde{\bold{L}}^-(n) + \mathbb{P}\left(\tau_{-n}<\infty, J_{\tau_{-n}} \right) \mathbb{P}\left(\tau_n<\infty,  J_{\tau_{ n}} \right) \widetilde{\bold{L}},
		\end{equation*}\end{linenomath*}
		or equivalently
		\begin{linenomath*}\begin{equation*}
				\widetilde{\bold{L}}^-(n) =\Bigl[\bold{I} - \mathbb{P}\left(\tau_{-n}<\infty, J_{\tau_{-n}} \right) \widetilde{\bold{G}}^n \Bigr] \widetilde{\bold{L}}.
		\end{equation*}\end{linenomath*}
		Now, although we do not discuss in much details here the definition and probabilistic interpretation of the matrix $\widetilde{\bold{R}}$, \cite{IVAN2} explain that the matrix $\widetilde{\bold{R}}^n$ comprises of $i,j$-th elements representing the expected number of visits to level $n\geqs 0$ in phase $j$ before the first return to the level $0$, given $X_0=0$ and $J_0=i$. Hence, using this interpretation, we observe that
		\begin{linenomath*}\begin{equation*}
				\widetilde{\bold{G}}^n\widetilde{\bold{L}}=\mathbb{E}\left(\widetilde{\bold{L}}(n,\infty)\right) = \widetilde{\bold{L}}\widetilde{\bold{R}}^n
		\end{equation*}\end{linenomath*}
		and therefore, straightforward calculations show that Eq.\,\eqref{twosidedsingular} holds for $\widetilde{\bold{V}}(n) = \widetilde{\bold{L}}^-(n)$
		as long as this matrix	is invertible for all $n\geqs 0$. Note that this can easily be verified by employing the same argument as in (i) of Remark \ref{Reminv} for $n \leqs 0$ and considering $\boldsymbol{\Pi}_{-n}$ instead of $\boldsymbol{\Pi}_n$.

		To prove Eq.\,\eqref{WVofn}, we use similar arguments as \cite{IVAN2} and employ the Markov property to obtain
		\begin{linenomath*}\begin{equation*}
		\widetilde{\bold{L}}^-(n+1)=\widetilde{\bold{L}}^-(n) + \widetilde{\bold{M}}(n)\widetilde{\bold{R}}^n,
		\end{equation*}\end{linenomath*}
	and, in particular, $\widetilde{\bold{L}}^-(1) = \widetilde{\bold{M}}(0)$, from which the result follows directly.
	
		Finally, to prove the transform in Eq.\,\eqref{WVofn2}, we again follow the methodology of \cite{IVAN2} and first note that by conditioning on the first time period, for $n \geqs 1$, we have
		\begin{linenomath*}\begin{equation}\label{eqM1}
				\widetilde{\bold{M}}(n)=\sum_{m=-1}^n\widetilde{\bold{A}}_{-m} \widetilde{\bold{M}}(n-m),
		\end{equation}\end{linenomath*}
	whilst, for $n=0$, it follows that
	 \begin{linenomath*}\begin{equation}\label{eqM2}
	 \widetilde{\bold{M}}(0)=\bold{I}+\widetilde{\bold{A}}_1\widetilde{\bold{M}}(1)+\widetilde{\bold{A}}_0\widetilde{\bold{M}}(0).
	 	\end{equation}\end{linenomath*}
	Taking transforms on both sides of Eq.\,\eqref{eqM1} and noting the above expression for $\widetilde{\bold{M}}(0)$, after some algebraic manipulations (see Appendix), we obtain

	\begin{linenomath*}\begin{align}\label{eqMTrans}
		\sum_{n=0}^\infty z^n\widetilde{\bold{M}}(n) &=\bold{I}-z^{-1}\widetilde{\bold{A}}_1\widetilde{\bold{M}}(0) +\F(z)\sum_{k=0}^\infty z^k	 \widetilde{\bold{M}}(k) \notag \\
		&= \bold{I}-z^{-1}\widetilde{\bold{R}} +\F(z)\sum_{k=0}^\infty z^k	 \widetilde{\bold{M}}(k),
	\end{align}\end{linenomath*}
where in the last equality we have use the probabilistic interpretation of $\widetilde{\bold{R}}$ to note that $\widetilde{\bold{R}} = \widetilde{\bold{A}}_1 \widetilde{\bold{L}}^-(1) = \widetilde{\bold{A}}_1 \widetilde{\bold{M}}(0)$. The result follows directly by solving the above expression for the transform and holds as long as $\bold{I}-\F(z)$ is invertible.
\end{proof}

\noindent Although the result of Theorem \ref{ThmTwoSideUpSing} is clearly more general than that of Theorem \ref{ThmTwoSideUp}, as it does not require invertibility of $\widetilde{\bold{A}}_1$, it deviates from the well known form and methodology of scale matrices (functions) seen throughout the literature. As such, since the purpose of this paper is to demonstrate and derive the fully discrete analogue of the well known `scale theory' for MACs, we will assume the invertibility of $\widetilde{\bold{A}}_1$ throughout the rest of this paper but point out that all the following results could also be generalised to the arbitrary case (see \cite{IVAN2} for more details of such results in the continuous-time setting).

At this point it is natural to consider the corresponding downward exit problems (one and two-sided). However, in order to do this we must first discuss some fluctuation problems for the associated `reflected' MAC process which is discussed in the following section.

\section{Exit Problems For Reflected MACs}
In this section, we deviate from the basic MAC described above and consider the associated two-sided reflection of the process $\{X_n\}_{n \geqs 0}$ with respect to a strip $[-d, 0]$ with $d > 0$. The choice of strip is purely for notational convenience and can easily be converted to the general strip $[-b,a]$ by shifting the process appropriately. The main result of this section is given in Theorem \ref{ThmPGR} which is interesting in its own right, but is also used to derive the aforementioned downward exit problems of the original (un-reflected) MAC.

Following the same line of logic as in \cite{IVAN1}, let us  define the reflected process by
\begin{linenomath*}\begin{equation*}
\label{eqH}
H_n = X_n + R^-_n - R^+_n,
\end{equation*}\end{linenomath*}
where $R^-_n$ and $R^+_n$ are known as regulators for the reflected process at the barriers $-d$ and $0$, respectively, which ensure that the process $\{H_n\}_{n \geqs 0}$ remains within the strip $[-d,0]$ for all $n \in \mathbb{N}$. Note that in continuous-time and space, the reflected process $\{H_n\}_{n \geqs 0}$ corresponds to the solution of the so-called Skorohod problem (see \cite{KRUK}). By the construction of $\{H_n\}_{n \geqs 0}$, it is clear that $\{R^-_n\}_{n \geqs 0}$ and $\{R^+_n\}_{n \geqs 0}$ are both non-decreasing processes, with $R^-_0 = R^+_0 = 0$, when $X_0$ in $[-d,0]$, which only increase during periods when $H_n = -d$ and $H_n = 0$, respectively. Moreover, since $\{X_n\}_{n \geqs 0} $ is `spectrally negative' the upward regulator $\{R^+_n\}_{n \geqs 0}$ increases by at most one per unit time.

Now, let us denote by $\rho_k$, the right inverse of the regulator $\{R^+_n\}_{n \geqs 0}$, defined by
\begin{linenomath*}\begin{equation}
\label{eqinv}
\rho_k = \inf\{ n \geqs 0 : R^+_n > k \} = \inf\{ n \geqs 0 : R^+_n = k + 1 \},
\end{equation}\end{linenomath*}
such that $R^+_{\rho_k} = k + 1$. Then, since an increase in $\{R^+_n\}_{n \geqs 0}$ only occurs whilst $H_n = 0$, it follows that $H_{\rho_k} = 0$ and thus, $R^-_{\rho_k} = (k+1) - X_{\rho_k}$. Hence, by the strong Markov property of $\{X_n\}_{n \geqs 0}$, we have that $\left\{ \left(R^-_{\rho_k}, J_{\rho_k} \right) \right\}_{k \geqs 0}$ is itself a MAC with random initial position $(R^-_{\rho_0}, J_{\rho_0})$ when $X_0 \in [-d,0]$ and non-negative jumps within the level process $\{R^-_{\rho_k}\}_{k \geqs 0}$. Thus, in a similar way as for the original MAC $(X, J)$, we can define its p.g.m.\,, given $X_0 = 0$, by
\begin{linenomath*}\begin{equation}
\label{eqPGF}
\mathbb{E} \bigl( z^{R^-_{\rho_k}}; J_{\rho_k} \bigr) = \bigl( \widetilde{\bold{F}}^*(z) \bigr)^{k+1}, \quad \widetilde{\bold{F}}^*(z) := \mathbb{E} \bigl( z^{R^-_{\rho_{_0}}}; J_{\rho_{_0}} \bigr).
\end{equation}\end{linenomath*}
\begin{Remark}
\label{RemTimes}
In the continuous case, $X_0 = 0$ is a regular point on $(0, \infty)$ and thus, it follows that $\rho_{_0} = 0$ a.s.\,and thus $\mathbb{E} \bigl( z^{R^-_{\rho_{_0}}}; J_{\rho_{_0}} \bigr) = \bold{I}$ (see \cite{IVAN1} for details). However, in the fully discrete set-up, we have already mentioned that $R^-_{\rho_0}$ is random for $X_0=0$ and is due to the possibility of the process experiencing a negative jump in the first time period such that $\rho_{_0} \neq 0$. Moreover, the process may drop below the lower level $-d$ (resulting in a jump in $\{R^-_n\}_{n \geqs 0}$) before the stopping time $\rho_{_0}$, and justifies the choice of the p.g.m.\,$\mathbb{E} \bigl( z^{R^-_{\rho_{_0}}}; J_{\rho_{_0}} \bigr)$ above, compared to $\mathbb{E} \bigl( z^{R^-_{\rho_{_1}}}; J_{\rho_{_1}} \bigr)$  in the continuous case (see \cite{IVAN1}). On the other hand, we note that if $X_0=1$, then $\mathbb{E} _1\bigl( z^{R^-_{\rho_{0}}}; J_{\rho_{0}} \bigr)=\bold{I}$, since $R^+_0=1$, and thus  $\rho_0=0$. The latter observation will play a crucial role in analysing the distribution of $(R^-_{\rho_0}, J_{\rho_0})$, which is given in the following theorem in terms of the second ${v}$-scale matrix, denoted $\Z_{v}$, and defined for $z \in (0,1]$ and $v \in (0,1]$, by
\begin{linenomath*}\begin{equation}
\label{eqZMat}
\widetilde{\bold{Z}}_{v}(z,n) = z^{-n} \Bigl[\bold{I} + \sum_{k=0}^n z^k\, \widetilde{\bold{W}}_{v}(k) \bigl(\bold{I} - {v}\widetilde{\bold{F}}(z) \bigr)   \Bigr],
\end{equation}\end{linenomath*}
with $\Z_v(z, 0) = \bold{I}$, for all $z \in (0,1]$ and $v \in (0,1]$ and $\Z_1(z,n) =: \Z(z, n)$.
\end{Remark}

\begin{Theorem}
\label{ThmPGR}
For $z \in (0,1]$, such that $z \notin \Lambda(\G)$, and $x \in [-d,1]$ it holds that $\Z(z, d+1)$ is invertible and
\begin{linenomath*}\begin{equation}
\label{eqPGR}
\mathbb{E}_x\bigl(z^{R^-_{\rho_{_0}}} ; J_{\rho_{_0}} \bigr) = \Z(z, d+x)\Z(z, d+1)^{-1},
\end{equation}\end{linenomath*}
where $\widetilde{\bold{Z}}(z, n)$ is defined by Eq.\,\eqref{eqZMat}.
\end{Theorem}
\begin{proof} The proof of this theorem actually follows a similar line of logic as the proof of Theorem \ref{TheoremOcc} however, due to the nature of the reflected process, the calculations require greater attention.
	
	First note that since $H_{\rho_{k}} = 0$ for each $k \in \mathbb{N}$, we have $X_{\rho_{k}} = k+1-R^-_{\rho_{k}}$ and thus $\{(X_{\rho_k}, J_{\rho_k})\}_{k \geqs 0}$ is a MAC having unit (upward) drift and downward jumps described by $\{R^-_{\rho_k}\}_{k \geqs 0}$ with random `initial' position $X_{\rho_0} = 1 - R^-_{\rho_0}$. Moreover, its occupation mass in the bivariate state $(y,j) \in \mathbb{Z} \times E$ is defined by $\widetilde{L}^*(y,j,\infty) = \sum_{k = 0}^\infty1_{\left( X_{\rho_k} = y, J_{\rho_k}=j\right)}
$
and thus, from the occupation mass formula in Eq.\,\eqref{occdens}, we have
\begin{linenomath*}\begin{equation*}
\sum_{k=0}^\infty z^{-X_{\rho_k}}1_{(J_{\rho_k} = j)} = \sum_{m \in \mathbb{Z}} z^{-m}\widetilde{L}^*(m,j,\infty).
\end{equation*}\end{linenomath*}
Taking expectations on both sides of this expression, conditioned on the initial state $X_0 = x \in [-d, 1]$, and writing in matrix form yields
\begin{linenomath*}\begin{equation}
\label{eqOCC}
\sum_{k=0}^\infty \mathbb{E}_x\left( z^{-X_{\rho_k}}; J_{\rho_k} \right) = \sum_{m \in \mathbb{Z}} z^{-m}\widetilde{\bold{L}}_x^*(m,\infty),
\end{equation}\end{linenomath*}
where $\widetilde{\bold{L}}_x^*(m,\infty)$ is the infinite-time occupation matrix with $i,j$-th element given by \\ $\bigl(\widetilde{\bold{L}}_x^*(m,\infty)\bigr)_{ij} = \mathbb{E}_x\bigl( \widetilde{L}^*(m,j,\infty) \mid J_0 = i\bigr)$.

Let us now treat  the left-hand side and right-hand side of Eq.\,\eqref{eqOCC} separately. Firstly, using the fact that $X_{\rho_{k}} = k+1-R^-_{\rho_{k}}$, along with the strong Markov and Markov additive properties of $\{R^-_{\rho_k}\}_{k \geqs 0}$, the l.h.s.\,of Eq.\,\eqref{eqOCC} can be re-written in the form
\begin{linenomath*}\begin{eqnarray}
\label{eqFirstTerm}
\sum_{k=0}^\infty \mathbb{E}_x\left( z^{-X_{\rho_k}}; J_{\rho_k} \right) &= &\sum_{k=0}^\infty z^{-(k+1)}\mathbb{E}_x\bigl( z^{R^-_{\rho_k}}; J_{\rho_k} \bigr) \notag \\
&=& \sum_{k=0}^\infty z^{-(k+1)}\mathbb{E}_x\bigl( z^{R^-_{\rho_{_0}}}; J_{\rho_{_0}} \bigr)\mathbb{E}\bigl( z^{R^-_{\rho_{k-1}}}; J_{\rho_{k-1}} \bigr) \notag \\
&= &\mathbb{E}_x\bigl( z^{R^-_{\rho_{_0}}}; J_{\rho_{_0}} \bigr) \sum_{k=0}^\infty z^{-(k+1)}\bigr( \F^*(z)\bigr)^k \notag \\
&= &\mathbb{E}_x\bigl( z^{R^-_{\rho_{_0}}}; J_{\rho_{_0}} \bigr) z^{-1} \bigl( \bold{I} - z^{-1}\F^*(z)\bigr)^{-1},
\end{eqnarray}\end{linenomath*}
for all  $z \in (0,1]$ such that $z > (\rho(\F^*(z))$. We note that since $\{(X_{\rho_k}, J_{\rho_k})\}_{k \geqs 0}$ is a MAC, it holds that $\mathbb E(z^{-X_{\rho_k}};J_{\rho_k})=\bigl(\mathbb E(z^{-X_{\rho_0}};J_{\rho_0})\bigr)^{k+1}$. Now, let us define $\overline{\tau}_1=\inf\{ \rho_k\geqs 0:X_{\rho_k}=1\}$ and $\overline {\bold{G}}$ to be the probability transition matrix such that $\mathbb P(\overline{\tau}_1 < \infty, J_{\overline{\tau}_1})=\overline {\bold{G}}$, which is sub-stochastic, (implying $\rho(\overline {\bold{G}}) < 1$) in the case of killing or no killing and negative drift. Then, based on similar arguments as those discussed in the proof of Theorem \ref{ThmTwoSideUp}, since the eigenvalues of $\overline {\bold{G}}$ coincide with the roots of $\bold I- \mathbb E(z^{-X_{\rho_0}};J_{\rho_0})=(\bold  I -z^{-1} \widetilde{\bold{F}}^*(z))$, then we conclude that $\bold  I -z^{-1} \widetilde{\bold{F}}^*(z)$ is invertible for $z \in (\rho(\overline {\bold{G}}),1]$. In fact, since $\{X_n\}_{n \geqs 0}$ is an upward skip-free process, it follows that  $\overline{\tau}_1=\tau_1$ for $X_0 \in [-d,1]$, which implies  $\overline {\bold{G}}=\widetilde {\bold{G}}$, and thus   $\bold  I -z^{-1} \widetilde{\bold{F}}^*(z)$ is invertible for $z \in (\rho(\widetilde {\bold{G}}), 1]$. Hence, by applying the same analytic continuation argument as in Theorem \ref{ThmTwoSideUp}, the above expression holds for $z\in (\rho(\widetilde {\bold{G}}),1)$.

Now, for the r.h.s.\,of Eq.\,\eqref{eqOCC}, let us introduce the matrix quantity $\widetilde{\bold{C}}_{-y}$ whose individual $i,j$-th elements denote the probability of the process $\{X_n\}_{n \geqs 0}$ first hitting some level $-y < 0 $ from initial states $X_0 = 0$ and $J_0 = i$, and then hitting the upper level $(d+1) - y$ whilst $J_n = j$, such that
\begin{linenomath*}\begin{align}
\label{eqC}
\widetilde{\bold{C}}_{-y} &= \mathbb{P}\bigl(\tau_{-y}<\infty, J_{\tau_{-y}} \bigr)\mathbb{P}_{-y}\bigl(\tau_{d+1-y}< \infty, J_{\tau_{d+1-y}}\bigr) \notag \\
&=\mathbb{P}\bigl(\tau_{-y}<\infty,J_{\tau_{-y}} \bigr)\mathbb{P}\bigl(\tau_{d+1}<\infty,J_{\tau_{d+1}} \bigr)  =\mathbb{P}\bigl(\tau_{-y}<\infty,J_{\tau_{-y}} \bigr)\widetilde{\bold{G}}^{d+1}.
\end{align}\end{linenomath*}
Using this quantity, it is possible to show that for $X_0 = x \in [-d,1]$
\begin{linenomath*}\begin{equation*}
\label{eqCMAT}
\widetilde{\bold{L}}^*_x(m, \infty) = \left[ 1_{(m> 0)} \mathbb{P}_x\big(\tau_m < \infty,  J_{\tau_m}\big) +  1_{(m \leqs 0)} \widetilde{\bold{C}}_{m - (d+1) - x}\right] \sum_{k =0}^\infty \left(\widetilde{\bold{C}}_{-(d+1)}\right)^k.
\end{equation*}\end{linenomath*}
To see this, note that $\widetilde{L}^*(m,j,\infty)$ corresponds to the (local) time points $\rho_k$ (increases in $\{R^+_n\}_{n \geqs 0}$) such that $X_{\rho_k} = m$ and $J_{\rho_k} = j$,  or alternatively, time points $k \geqs 0$ for which $\{R^+_n\}_{n \geqs 0}$ is increasing and $X_k = m$ and $J_k = j$. Then, for $m > 0$, the first increase of $\widetilde{L}^*(m,j,\infty)$ is at $\tau_m$, otherwise, for $m \leqs 0$, $\{X_n\}_{n \geqs 0}$ has to first visit the state (level) $m-(d+1)$ to ensure that at the next time the process $\{X_n\}_{n \geqs 0}$ visits the level $m < 0$, the `reflected process' $\{H_n\}_{n \geqs 0}$ was at its upper boundary in the previous time period ($H_{n-1} = 0$), resulting in an increase of $\{R^+_n\}_{n \geqs 0}$. Every subsequent increase of $\widetilde{L}^*(m,j,\infty)$ is obtained in a similar way. Thus, the above equation follows by application of the strong Markov and Markov additive properties.

Taking transforms on both sides of the above equation, it yields
\begin{linenomath*}\begin{align}
\label{eqC1}
\sum_{m \in \mathbb{Z}} z^{-m} \widetilde{\bold{L}}^*_x(m, \infty) &= \Bigl( \sum_{m=1}^\infty z^{-m}\mathbb{P}_x\big(\tau_m < \infty, J_{\tau_m}\big) +  \sum_{m = -\infty}^{0} z^{-m}\widetilde{\bold{C}}_{m - (d+1) - x}\Bigr) \sum_{k =0}^\infty \left(\widetilde{\bold{C}}_{-(d+1)}\right)^k \notag  \\
  &= \Bigl( \sum_{m=1}^\infty z^{-m}\widetilde{\bold{G}}^{m-x} +  \sum_{m = -\infty}^{0} z^{-m}\mathbb{P}\big(\tau_{m-(d+1)-x}< \infty, J_{\tau_{m-(d+1)-x}} \big) \widetilde{\bold{G}}^{d+1}
 \Bigr) \notag
 \\ &\quad \times \left(\bold{I} - \widetilde{\bold{C}}_{-(d+1)}\right)^{-1},
\end{align}\end{linenomath*}
where we have used the fact that $\sum_{k =0}^\infty \bigl(\widetilde{\bold{C}}_{-(d+1)}\bigr)^k = \bigl(\bold{I} - \widetilde{\bold{C}}_{-(d+1)}\bigr)^{-1}$ in the presence of killing, since $ \widetilde{\bold{C}}_{-(d+1)}$ is a sub-stochastic matrix and thus, its Perron-Frobenius eigenvalue is less than 1. Now, the first term inside the brackets of the last expression is clearly equivalent to $-\bigl( \bold{I} - z\widetilde{\bold{G}}^{-1}\bigr)^{-1}\widetilde{\bold{G}}^{-x}$ for all $z \in (\rho(\widetilde{\bold{G}}), 1]$, whilst by the change of variable $k = m - (d+1) - x$, the second term within the brackets becomes
\begin{linenomath*}\begin{align*}
\label{eqSUM}
z^{-(d+1) - x} \sum_{k = -\infty}^{-(d+1+x)} z^{-k}\mathbb{P}\big(\tau_k < \infty, J_{\tau_{k}} \big) \widetilde{\bold{G}}^{d+1} = z^{-(d+1) - x} \sum_{m = d+1 + x}^\infty z^{m}\mathbb{P}\big(\tau_{-m} < \infty, J_{\tau_{-m}} \big) \widetilde{\bold{G}}^{d+1},
\end{align*}\end{linenomath*}
and thus, after some algebraic manipulations (see Appendix), Eq.\eqref{eqC1} can be re-written as
\begin{linenomath*}\begin{equation}
\label{eqC1.1}
\sum_{m \in \mathbb{Z}} z^{-m} \widetilde{\bold{L}}^*_x(m, \infty)
   =z^{-1}\Z(z, d+x)(\bold{I} - \F(z))^{-1} \W(d+1)^{-1},
\end{equation}\end{linenomath*}
where $\Z(z, n)$ is defined in Eq.\,\eqref{eqZMat}. Finally, by combining Eqs.\,\eqref{eqFirstTerm} and \eqref{eqC1.1}, we obtain for  $z \in (\rho( \widetilde{\bold{G}}),1)$
\begin{linenomath*}\begin{equation}
\label{eqResult}
\mathbb{E}_x\bigl(z^{R^-_{\rho_{_0}}}; J_{\rho_{_0}} \bigr) \bigl( \bold{I} -z^{-1} \F^*(z) \bigr)^{-1}= \Z(z, d+x)(\bold{I} - \F(z))^{-1} \W(d+1)^{-1}.
\end{equation}\end{linenomath*}
To complete the proof, it remains  to determine the form of the matrix $\F^*(z)$. To do this, let $x = 1$ into the above expression which, after using the fact that $\mathbb{E}_1\bigl(z^{R^-_{\rho_{_0}}}; J_{\rho_{_0}} \bigr)=\bold{I}$ since in this case $\rho_0 = 1$ and taking inverses on both sides, gives
\begin{linenomath*}\begin{equation*}
 \bold{I} -z^{-1} \F^*(z)  = \W(d+1)(\bold{I} - \F(z))^{-1} \Z(z, d+1)^{-1}.
\end{equation*}\end{linenomath*}
Note that this expression shows that $ \Z(z, d+1)$ is an invertible matrix as long as $\widetilde{\bold{W}}(d+1)$ is invertible and after solving w.r.t.\,$\F^*(z)$ also gives
\begin{linenomath*}\begin{equation}
\label{fstar}
\F^*(z)=z\left[\bold{I}-\W(d+1)(\bold{I} - \F(z))^{-1} \Z(z, d+1)^{-1}\right].
\end{equation}\end{linenomath*}
The result follows by substituting the above expression for  $\F^*(z)$ back into Eq.\,\eqref{eqResult}, re-arranging and employing analytic continuation in a similar way as previous.

\end{proof}

\begin{Remark}
We point out that setting $X_0=x=0$ in the result of Theorem \ref{ThmPGR}, gives an equivalent representation for $\widetilde{\bold{F}}^*(z)$ in terms of the $\Z$ scale matrix only, i.e.,
 \begin{linenomath*}\begin{equation*}
 		\label{fstar1}
 		\F^*(z)=\Z(z, d)\Z(z, d+1)^{-1}.
 \end{equation*}\end{linenomath*}
Moreover, we note that based on its definition, it is also possible to use the recursive relation of $\W(\cdot)$, given in Corollary \ref{CorW}, to obtain explicit values of $\Z(z, \cdot)$.
\end{Remark}

\noindent Although the result of Theorem \ref{ThmPGR} is interesting in its own right, its main importance in this paper is as a stepping stone for proving a similar result for the associated one-sided reflected process (see Section \ref{OSref} below) and consequently, the two-sided and one-sided (as a limiting case) downward exit problems for the original (non-reflected) MAC.
\subsection{One-Sided Reflection}
\label{OSref}
As discussed in the previous section, the downward exit problems can be solved using an auxiliary result for the one-sided (lower) reflected process. As such, let us define
\begin{linenomath*}\begin{equation*}
Y_n = X_n + R^{-b}_n,
\end{equation*}\end{linenomath*}
where $R^{-b}_n = -b-(-b \wedge \underline{X}_n)$ with $\underline{X}_n = \inf_{k \leqs n} \{X_k\}$, denotes a lower reflecting barrier at the level $-b \leqs 0$. Note that this is equivalent to shifting the two-sided reflected process of the previous section and letting the upper reflecting barrier tend to infinity. Then, by direct application of Theorem \ref{ThmPGR} we get the following  corollary.
\begin{Corollary}
\label{CorOneSide}
For $X_0 = 0$, $z \in (0,1]$ such that $z \notin \Lambda(\G)$, $a > 0$ and $b \geqs 0$, it holds that

\begin{linenomath*}\begin{equation}
\label{eqOneSidedRef}
\mathbb{E}\Bigl(z^{R^{-b}_{\,\tau_a}};J_{\tau_a} \Bigr) = \Z(z, b)\Z(z, a+b)^{-1},
\end{equation}\end{linenomath*}
\end{Corollary}
\begin{proof}
Note that if we set $d = (a-1)+b$ in Theorem \ref{ThmPGR}, then $\{(H_n + (a-1), R^-_n)\}_{n \geqs 0}$ up to time $\rho_{_0}$ coincides with $\{(Y_n, R^{-b}_n)\}_{n \geqs 0}$ up to time $\tau_{a}$, given that $H_0 + (a-1) = Y_0$. Hence, the result follows directly from Theorem \ref{ThmPGR} with $x = -(a-1)$.
\end{proof}

\section{Downward Exit Problems}

For the one and two-sided downward exit problems, we are interested in the events $\{\tau^-_{-b} < \infty\}$ and $\{ \tau_{-b}^- < \tau^+_a\}$, respectively. Unlike the upward exit, due to the possibility of downward jumps in the MAC, the stopping time $\tau^-_{-b}$ is not necessarily equivalent to the first hitting time of the level $-b < 0$, i.e., $\tau^-_{-b} \neq \tau_{-b}$. It is for this reason that we cannot employ the Markov type structure seen for the upward exit identities and, instead, rely on the results of the reflected processes of the previous section.

Although it would appear easier to derive in the first instance, it turns out that the one-sided downward exit problem can easily be obtained as a limiting case of the related two-sided case and as such, is considered in the following.

\subsection{Two-Sided Exit Downward - $\{ \tau_{-b}^- < \tau_a^+\}$ }
\label{SecTwoSide}

For the two-sided downward exit problem, we are interested in the time of exiting the fixed `strip', $[-b,a]$, such that $\{ \tau_{-b}^- < \tau^+_a\}$. Using the result for the transform of the downward regulator for the one-sided reflected process, we obtain the following corollary.

\begin{Corollary}
\label{CorTwoSide}
For $z \in [0,1]$ such that $z \notin \Lambda(\G)$, it holds that for any $a,b > 0$, we have
\begin{linenomath*}\begin{equation*}\resizebox{1.01\hsize}{!}{$
\label{EqTwoDown}{
\mathbb{E}\Bigl( z^{-X_{\tau^-_{-b}}}; \tau^-_{-b} < \tau^+_a, J_{\tau^-_{-b}}\Bigr) = z^{b-1} \bigl[\Z(z, b-1)-\W(b)\W(a+b)^{-1}\Z(z, a+b-1)\bigr] .}$}\notag \\
\end{equation*}\end{linenomath*}
\end{Corollary}

\begin{proof}
Consider the one-sided reflected process of Section \ref{OSref}. Then, by the strong Markov and Markov additive properties, it follows that for $b > 0$, we have
\begin{linenomath*}\begin{equation*}
\mathbb{E} \Bigl(z^{R^{-(b-1)}_{\tau^+_a}}; J_{\tau^+_a} \Bigr) = \mathbb{P}\bigl(\tau^+_a < \tau^-_{-b}; J_{\tau^+_a} \bigr) + \mathbb{E} \Bigl( z^{-(b-1)-X_{\tau^-_{-b} }}; \tau^-_{-b}<\tau^+_a,J_{\tau^-_{-b}} \Bigr) \mathbb{E}\Bigl(z^{R^{0}_{\tau^+_{a+b-1}}}; J_{\tau^+_{a+b-1}} \Bigr).
 \end{equation*}\end{linenomath*}
Re-arranging this expression and using the identities of Theorem \ref{ThmTwoSideUp} and Corollary \ref{CorOneSide}  the result follows immediately.
\end{proof}

\subsection{One-Sided Exit Downward}

For the one-sided exit problem, we are now interested in the event that of down-crossing the level $-b < 0$, whilst the upward movement of the MAC is un-restricted, i.e., $\{\tau^-_{-b} < \infty\}$ which, as already mentioned, can be viewed as a limiting case of the corresponding two-sided problem as $a \rightarrow \infty$. In fact, this is the argument used to obtain the following one-sided downward exit identity.

\begin{Corollary}
\label{CorOneSideDown}
Assume we are not in the case of no-killing and zero drift, i.e., it is not true that both $v = 1$ and $\kappa'(1) = 0$. Then, $\widetilde{\bold{L}}$ is invertible and, for $z \in (0,1]$ such that $z \notin \Lambda(\G)$ and $b > 0$, we have
\begin{linenomath*}\begin{equation}
\label{OneSideDown}
\mathbb{E}\Bigl(z^{-X_{{\tau^-_{-b}}}}; J_{_{\tau^-_{-b}}} \Bigr) = z^{b-1}\Bigl[\Z (z, b-1) - z \W(b)\widetilde{\bold{L}}^{-1}\bigl(\bold{I} -z \G^{-1}\bigr)^{-1}\widetilde{\bold{L}}(\F(z) - \bold{I} )\Bigr] .
\end{equation}\end{linenomath*}
\end{Corollary}

\noindent

\begin{proof}
Firstly, the invertibility of $\widetilde{\bold{L}}$ follows from Remark \ref{RemInvert}, for which it cannot hold that both $v=1$ and $\kappa'(1) = 0$. On the other hand, Eq.\,\eqref{OneSideDown} follows from taking the limit of the two-sided case (see Corollary \ref{CorTwoSide}) as the upper barrier tends to infinity, i.e., $a \rightarrow \infty$.
In order to evaluate the value of the limit of $\W(b)\W(a+b)^{-1}\Z(z, a+b-1)$ as $a \rightarrow \infty$, note that by the definition of the scale matrix $\Z(z, n)$, and using Eq.\,\eqref{eqWTran},  it follows that
\begin{linenomath*}\begin{eqnarray*}
\Z(z, a+b-1) &= &z^{-(a+b-1)}\Bigl( \bold{I} + \sum_{k=0}^{a+b-1}z^k \W(k)\bigl( \bold{I} - \F(z)\bigr)\Bigr) \\
&=& z^{-(a+b-1)}\sum_{k=a+b}^\infty z^k \W(k)\bigl(\F(z) - \bold{I}\bigr) \\
&=& \sum_{n=1}^\infty z^n \W(n+a+b-1)\bigl(\F(z) - \bold{I}\bigr).
\end{eqnarray*}\end{linenomath*}
Moreover, by using the fact that $\W(n) = \G^{-n} \widetilde{\bold{L}}(n)$ (see Theorem \ref{ThmTwoSideUp}), multiplication of the above expression by $\W(a+b)^{-1}$ on the left yields
\begin{linenomath*}\begin{equation*}
\W(a+b)^{-1}\Z(z, a+b-1) = \widetilde{\bold{L}}^{-1}(a+b)\sum_{n=1}^\infty z^n \G^{-(n-1)} \widetilde{\bold{L}}(n + a + b -1) \bigl(\F(z) - \bold{I}\bigr),
\end{equation*}\end{linenomath*}
which, after taking $a \rightarrow \infty$ and using dominated convergence theorem, gives
\begin{linenomath*}\begin{eqnarray*}
\lim_{a \rightarrow \infty }\W(a+b)^{-1}\Z(z, a+b-1) &=& \widetilde{\bold{L}}^{-1}z\sum_{n=0}^\infty \bigl(z \G^{-1}\bigr)^n \widetilde{\bold{L}} \bigl(\F(z) - \bold{I}\bigr) \\
&=&\widetilde{\bold{L}}^{-1}z\bigl(\bold{I} -z\G^{-1} \bigr)^{-1} \widetilde{\bold{L}} \bigl(\F(z) - \bold{I}\bigr),
\end{eqnarray*}\end{linenomath*}
for $z \in (0, \gamma)$, where $\widetilde{\bold{L}}$ is the infinite time occupation mass matrix defined in Proposition \ref{PropDens}. Finally, by analytic continuation, it can be shown that the above holds for all $z \in (0,1]$ such that $z \notin \Lambda(\G)$ and thus, by
 taking the limit as $a \rightarrow \infty$ in Corollary \ref{CorTwoSide}, using the above expressions and re-arranging, we obtain the result.
\end{proof}

\begin{Remark}
	\noindent We point out once again that by explicitly imposing killing, Corollary \ref{CorTwoSide} and consequently Corollary \ref{CorOneSideDown} equivalently yield the following joint transforms for $v \in (0,1]$
	\begin{linenomath*}\begin{equation*}\resizebox{1.01\hsize}{!}{$
				\mathbb{E}\Bigl( v^{\tau^-_{-b}}z^{-X_{\tau^-_{-b}}}; \tau^-_{-b} < \tau^+_a, J_{\tau^-_{-b}}\Bigr) =z^{b-1} \bigl[\Z_v(z, b-1)-\W_v(b)\W_v(a+b)^{-1}\Z_v(z, a+b-1)\bigr],$}
	\end{equation*}\end{linenomath*}
and
\begin{linenomath*}\begin{equation}
\label{OneSideDown}
\mathbb{E}\Bigl(v^{\tau^-_{-b}} z^{-X_{{\tau^-_{-b}}}}; J_{_{\tau^-_{-b}}} \Bigr) = z^{b-1}\Bigl[\Z _v(z, b-1) - z \W_v(b)\widetilde{\bold{L}}_v^{-1}\bigl(\bold{I} -z \G_v^{-1}\bigr)^{-1}\widetilde{\bold{L}}_v(\F_v(z) - \bold{I} )\Bigr] .
\end{equation}\end{linenomath*}
where $\W_v(\cdot)$ and $\Z_v(z, \cdot)$ are defined as in Eqs.\,\eqref{eqW1} and \eqref{eqZMat}, respectively.
\end{Remark}

\section*{Acknowledgments}
This work is partially supported by the National Science Centre under the grant\linebreak
2021/41/B/HS4/00599 (2022-2026).
The authors are grateful to the anonymous referees for their constructive comments and suggestions that have improved the content and presentation of this paper.


\section*{Conflict of interest statement}
The authors have no conflicts of interest to declare that are relevant to the content of this article.

\section*{Data availability statement}
Data sharing not applicable to this article as no datasets were generated or analyzed during the
current study.



\section*{Appendix}
\renewcommand{\theequation}{A.\arabic{equation}}
\setcounter{equation}{1}

\begin{proof}[\textbf{Proof of Eq.\,\eqref{eqMTrans}}.]
	It follows from the results of Eq.\,\eqref{eqM1} and \eqref{eqM2}, that
	\begin{linenomath*}\begin{align*}
			\sum_{n=0}^\infty z^n\widetilde{\bold{M}}(n) &= \widetilde{\bold{M}}(0)+\sum_{n=1}^\infty z^n	\sum_{m=-1}^n\widetilde{\bold{A}}_{-m} \widetilde{\bold{M}}(n-m)\\ &=\bold{I}+\sum_{n=0}^\infty z^n	\sum_{m=-1}^n\widetilde{\bold{A}}_{-m} \widetilde{\bold{M}}(n-m) \\
			&= \bold{I}+\sum_{n=0}^\infty z^n	\sum_{k=0}^{n+1}\widetilde{\bold{A}}_{-(n-k)} \widetilde{\bold{M}}(k)\\
			&= \bold{I}+\sum_{n=0}^\infty z^n	\widetilde{\bold{A}}_{-n} \widetilde{\bold{M}}(0)+\sum_{n=0}^\infty z^n	\sum_{k=1}^{n+1}\widetilde{\bold{A}}_{-(n-k)} \widetilde{\bold{M}}(k)\\
			&= \bold{I}+\sum_{n=0}^\infty z^n	\widetilde{\bold{A}}_{-n} \widetilde{\bold{M}}(0)+\sum_{k=1}^\infty z^k	\sum_{n=k-1}^{\infty}z^{n-k}\widetilde{\bold{A}}_{-(n-k)} \widetilde{\bold{M}}(k)\\
			&=\bold{I}+\sum_{n=0}^\infty z^n	\widetilde{\bold{A}}_{-n} \widetilde{\bold{M}}(0)+  \sum_{i=-1}^{\infty}z^i \widetilde{\bold{A}}_{-i}\sum_{k=1}^\infty z^k	 \widetilde{\bold{M}}(k)\\
			&=\bold{I}-z^{-1}\widetilde{\bold{A}}_1\widetilde{\bold{M}}(0) +\F(z)\sum_{k=0}^\infty z^k	 \widetilde{\bold{M}}(k),
	\end{align*}\end{linenomath*}
where, in the last equality, we have used the series definition of $\F(z)$ given in Eq.\,\eqref{eqFMat}.
\end{proof}
\begin{proof}[\textbf{Proof of Eq.\,\eqref{eqC1.1}}.]
To prove Eq.\,\eqref{eqC1.1}, first note that
\begin{linenomath*}\begin{equation*}\label{Ape1}
 \sum_{n=0}^k z^n \mathbb{P}\left(\tau_{-n}<\infty, J_{\tau_{-n}} \right) \LMat=\left[ \sum_{n=0}^\infty z^n \mathbb{P}\left(\tau_{-n}<\infty, J_{\tau_{-n}} \right)- \sum_{n=k+1}^\infty z^n\mathbb{P}\left(\tau_{-n}<\infty, J_{\tau_{-n}} \right)  \right]\LMat.
\end{equation*}\end{linenomath*}
Then, solving  Eq.\,\eqref{ap4} w.r.t.\,$ \sum_{n=0}^\infty z^n \mathbb{P}\bigl(\tau_{-n}<\infty, J_{\tau_{-n}} \bigr)\widetilde{\bold{L}}$ and substituting  into the above equation, we have
\begin{linenomath*}\begin{equation}\label{Ape2}
 \sum_{n=0}^k z^n \mathbb{P}\left(\tau_{-n}<\infty, J_{\tau_{-n}} \right) \LMat=\bigl(\bold{I} - \widetilde{\bold{F}}(z) \bigr) ^{-1}+ \bigl(\bold{I}-z^{-1}\widetilde{\bold{G}}^{-1}\bigl)^{-1} \LMat\,  - \sum_{n=k+1}^\infty  z^n \mathbb{P}\left(\tau_{-n}<\infty,  J_{\tau_{-n}} \right)\LMat.\quad
\end{equation}\end{linenomath*}
Now, at this point, consider the definition of the scale matrix, $\W(n)$, given in Eq.\,\eqref{ScaleMat1}. Multiplying this expression through by $z^n$ and summing from $0$ to $k$ on both sides, gives
\begin{linenomath*}\begin{equation}
\label{Ape3}
\sum_{n=0}^k z^n \mathbb{P}\left(\tau_{-n}<\infty, J_{\tau_{-n}} \right) \LMat =  \sum_{n=0}^k z^n \G^{-n} \LMat - \sum_{n=0}^k z^n \W(n),
\end{equation}\end{linenomath*}
and thus by equating the r.h.s.\,of Eqs.\,\eqref{Ape2} and \eqref{Ape3} and re-arranging, we obtain
\begin{linenomath*}\begin{align}\label{Ape4}
\sum_{n=k+1}^\infty  z^n \mathbb{P}\left(\tau_{-n}<\infty,  J_{\tau_{-n}} \right)&=\sum_{n=0}^k z^n \W(n)\LMat^{-1}+\bigl(\bold{I} - \widetilde{\bold{F}}(z) \bigr) ^{-1}\LMat^{-1}-\sum_{n=0}^k z^n \G^{-n} \notag \\
&\hspace{5mm}+ \bigl(\bold{I}-z\widetilde{\bold{G}}^{-1}\bigl)^{-1}  \notag\\
&=\sum_{n=0}^k z^n \W(n)\LMat^{-1} +\bigl(\bold{I} - \widetilde{\bold{F}}(z) \bigr) ^{-1}\LMat^{-1} \notag \\
&\hspace{5mm} - \bigl(\bold{I}-z\widetilde{\bold{G}}^{-1}\bigl)^{-1} \bigl(\bold{I}-(z\widetilde{\bold{G}}^{-1})^{k+1}\bigl) +\bigl(\bold{I}-z\widetilde{\bold{G}}^{-1}\bigl)^{-1}  \notag \\
&=\sum_{n=0}^k z^n \W(n)\LMat^{-1}+\bigl(\bold{I} - \widetilde{\bold{F}}(z) \bigr) ^{-1}\LMat^{-1} \notag \\
& \hspace{5mm}+ \bigl(\bold{I}-z\widetilde{\bold{G}}^{-1}\bigl)^{-1} (z\widetilde{\bold{G}}^{-1})^{k+1},
\end{align}\end{linenomath*}
which provides an expression for the second term of Eq.\,\eqref{eqC1}.  Thus, letting $k=d+x$ in the above expression and substituting into Eq.\,\eqref{eqC1}, we have that
\begin{linenomath*}\begin{align}\label{Ape5}
\sum_{m \in \mathbb{Z}} z^{-m} \widetilde{\bold{L}}^*_x(m, \infty) &= \Bigl[-\bigl( \bold{I} - z\widetilde{\bold{G}}^{-1}\bigr)^{-1}\widetilde{\bold{G}}^{-x}+  z^{-(d+1) - x} \Bigg(\sum_{n=0}^{d+x} z^n \W(n)\LMat^{-1}+\bigl(\bold{I} - \widetilde{\bold{F}}(z) \bigr) ^{-1}\LMat^{-1}
  \notag \\
 &\quad + \bigl(\bold{I}-z\widetilde{\bold{G}}^{-1}\bigl)^{-1} (z\widetilde{\bold{G}}^{-1})^{d+x+1}\Bigg)\widetilde{\bold{G}}^{d+1} \Bigr]\left(\bold{I} - \widetilde{\bold{C}}_{-(d+1)}\right)^{-1}\notag \\
&= z^{-(d+1) - x} \left[\sum_{n=0}^{d+x} z^n \W(n)+\bigl(\bold{I} - \widetilde{\bold{F}}(z) \bigr) ^{-1}\right]\LMat^{-1}\widetilde{\bold{G}}^{d+1} \left(\bold{I} - \widetilde{\bold{C}}_{-(d+1)}\right)^{-1}
\end{align}\end{linenomath*}
Now, setting $n=d+1$ in Eq. \,\eqref{ScaleMat1} and multiplying from the right by $\LMat^{-1}\widetilde{\bold{G}}^{d+1}$, yields
\begin{linenomath*}\begin{align*}\label{Ape6}
\W(d+1)\LMat^{-1}\widetilde{\bold{G}}^{d+1}&=\bold{I}-\mathbb{P}\big(J_{\tau_{-(d+1)}} \big)\widetilde{\bold{G}}^{d+1} \\
&=\bold{I} - \widetilde{\bold{C}}_{-(d+1)},
\end{align*}\end{linenomath*}
by the definition of $\widetilde{\bold{C}}_{-y}$ given in  Eq.\,\eqref{eqC} and thus, it follows that
\begin{linenomath*}\begin{equation*}\label{Ape7}
\left(\bold{I} - \widetilde{\bold{C}}_{-(d+1)}\right)^{-1}=\widetilde{\bold{G}}^{-(d+1)}\LMat\W(d+1)^{-1}.
\end{equation*}\end{linenomath*}
Finally, substituting the above equation into Eq.\,\eqref{Ape5}, we get that
\begin{linenomath*}\begin{align}\label{Ape8}
\sum_{m \in \mathbb{Z}} z^{-m} \widetilde{\bold{L}}^*_x(m, \infty) &= z^{-(d+1) - x} \left[\sum_{n=0}^{d+x} z^n \W(n)+\bigl(\bold{I} - \widetilde{\bold{F}}(z) \bigr) ^{-1}\right]\W(d+1)^{-1}\notag\\
&= z^{-1}z^{-(d+ x)} \left[\sum_{n=0}^{d+x} z^n \W(n)\bigl(\bold{I} - \widetilde{\bold{F}}(z) \bigr)+\bold{I}\right]\bigl(\bold{I} - \widetilde{\bold{F}}(z) \bigr) ^{-1}\W(d+1)^{-1}\notag\\
&=z^{-1}\Z(z, d+x)(\bold{I} - \F(z))^{-1} \W(d+1)^{-1}, \notag
\end{align}\end{linenomath*}
where the last equation follows from the definition of the $\Z$ scale matrix given in Eq.\,\eqref{eqZMat}.
\end{proof}

\end{document}